\documentclass[12pt]{amsart}
\usepackage{fullpage,amssymb}
\usepackage{mathrsfs}
\usepackage{xcolor}
\usepackage{algorithmicx}
\usepackage{algpseudocode}
\usepackage{tikz-cd}
\usepackage{bm}
\usepackage{bbold}

\usepackage{xy}
\xyoption{all}
\usepackage{color}   
\usepackage{hyperref}
\hypersetup{
    linktoc=all,     
    linkcolor=black,  
}

\newtheorem{theorem}{Theorem}[section]
\newtheorem{corollary}[theorem]{Corollary} 
\newtheorem{lemma}[theorem]{Lemma}
\newtheorem{proposition}[theorem]{Proposition}

\theoremstyle{definition}
\newtheorem{definition}[theorem]{Definition}
\newtheorem{remark}[theorem]{Remark}
\newtheorem{example}[theorem]{Example}
\newtheorem{problem}[theorem]{Problem}

\newtheorem{algorithm}[theorem]{Algorithm}

\newcommand{\kar} {{\rm char}}
\newcommand{\R}{{\mathbb R}}

\newcommand{\B}{{\mathcal B}}
\newcommand{\C}{{\mathbb C}}
\newcommand{\N}{{\mathbb N}}

\newcommand{\Q}{{\mathbb Q}}
\newcommand{\Z}{{\mathbb Z}}

\newcommand{\Mat}{\operatorname{Mat}}
\newcommand{\GL}{\operatorname{GL}}
\newcommand{\SL}{\operatorname{SL}}
\newcommand{\WP}{{\rm WP}}
\newcommand{\supp}{{\rm supp}}
\newcommand{\NP}{{\rm NP}}
\newcommand{\tsupp}{{\rm tsupp}}
\newcommand{\spa}{{\rm span}}
\newcommand{\ST}{{\rm ST}}

\newcommand{\VV}{{\mathbb V}}

\newcommand{\nullcone}{\mathcal{N}}
\newcommand{\Sym}{S}

\newcommand{\gitsymbol}{\mathbin{
  \mathchoice{/\mkern-6mu/}
    {/\mkern-6mu/}
    {/\mkern-5mu/}
    {/\mkern-5mu/}}}

\usepackage[toc]{appendix}

\title{Polystability in positive characteristic and degree lower bounds for invariant rings}
\author{Harm Derksen and Visu Makam}
\begin{document}

\maketitle
\begin{abstract}
We develop a representation theoretic technique for detecting closed orbits that is applicable in all characteristics. Our technique is based on Kempf's theory of optimal subgroups and we make some improvements and simplify the theory from a computational perspective. We exhibit our technique in many examples and in particular, give an algorithm to decide if a symmetric polynomial in $n$-variables has a closed $\SL_n$-orbit.

As an important application, we prove exponential lower bounds on the maximal degree of a system of generators of invariant rings for two actions that are important from the perspective of Geometric Complexity Theory (GCT). The first is the action of $\SL(V)$ on $\Sym^3(V)^{\oplus 3}$, the space of $3$-tuples of cubic forms, and the second is the action of $\SL(V) \times \SL(W) \times \SL(Z)$ on the tensor space $(V \otimes W \otimes Z)^{\oplus 5}$. In both these cases, we prove an exponential lower degree bound for a system of invariants that generate the invariant ring or that define the null cone.

\end{abstract}
 

\setcounter{tocdepth}{1}
\tableofcontents 

\section{Motivation}
We choose our ground field $K$ to be an algebraically closed field of characteristic $p$.\footnote{Our results will be targeted towards the case of $p > 0$, but many of our results are new even in the case of $p = 0$.} In this paper, we focus on two important problems with particular emphasis on positive characteristic -- how to determine whether an orbit is closed (a.k.a. polystability) and how to prove exponential degree lower bounds for invariant rings. We briefly discuss the motivation behind these problems and give context to the main contributions of this paper before proceeding to the main content. 

To begin, let us consider the following result: 

\begin{theorem} \label{thm:complete-homog-3vars}
Consider the action of $\SL(V)$ on $\Sym^3(V)$, the space of cubic forms, where $V$ is a $3$-dimensional vector space over $K$ with basis $x,y,z$. Consider the complete homogeneous symmetric polynomial $h_3(x,y,z) \in \Sym^3(V)$, i.e., $h_3(x,y,z)$ is the sum of all monomials of degree $3$. Then $h_3(x,y,z)$ is polystable (i.e, its $\SL(V)$ orbit is Zariski-closed) unless $p \in \{2,5\}$.
\end{theorem}

How does one go about proving such a result? What techniques do we have to determine whether an orbit is closed or not, especially in positive characteristic? Naively, one could try to get hold of the ideal of polynomials which vanish on the orbit and check if its zero locus contains a point outside the orbit. However, there seems to be no reasonable way to do this. One natural approach for this would be via generators of the invariant ring, but that is computationally infeasible even in the seemingly simple example of $h_3(x,y,z)$ above. 

In characteristic zero, one useful result is the Dadok-Kac criterion \cite{Dadok-Kac} (see \cite[Section~6]{DM-exp} for a generalization). Another approach used in literature is a criterion due to Kempf \cite[Corollary~4.5]{Kempf}, but this is again only applicable in characteristic zero.\footnote{For example, \cite{BI} uses \cite[Corollary~4.5]{Kempf} to prove that the matrix multiplication and unit tensors have closed orbits, which by the way also follows easily from the Dadok-Kac criterion.} But even in characteristic zero, the example of $h_3(x,y,z)$ above does not fall within the scope of either tool. Yet another tool at our disposal is the fact that any homogeneous polynomial with a non-vanishing discriminant has a closed orbit with a finite stabilizer (this holds in arbitrary characteristic). However, discriminants are very hard to compute, 
 and are very specific to the action of $\SL(V)$ on $\Sym^d(V)$ without much scope for generalizing to other actions. To summarize, while certain techniques for proving closedness of orbits exist in literature, they are quite limited in scope and severely lacking in their applicability in positive characteristic. 

Our motivation for investigating closed orbits in positive characteristic comes from the problem of degree bounds in invariant theory and in particular results on exponential degree lower bounds \cite{DM-exp}. The significance of degree bounds in invariant theory is best understood through the lens of computational complexity, and in particular the Geometric Complexity Theory (GCT) program.\footnote{ The GCT program is an algebro-geometric approach to the celebrated P vs NP problem.} An understanding of degree bounds is a first step in a large, extensive, and ambitious program put forth in \cite{GCTV} that aims to connect invariant theory and central problems in complexity at a fundamental level. In recent years, an alternate approach to algorithmic invariant theory using geodesic optimization techniques has emerged, see \cite{BFGOWW} and references therein. However, these new optimization techniques are manifestly a characteristic zero approach. With no promising alternative approach in positive characteristic, the algebraic approaches and in particular degree bounds find a renewed importance in positive characteristic.

In a previous paper \cite{DM-exp}, we proved exponential degree lower bounds for the generators of invariant rings for cubic forms and tensor actions in characteristic zero. The technique was based on the Grosshans principle and a major component was to prove that certain points (with significant symmetries) have closed orbits.\footnote{We had used (a generalization of) the Dadok-Kac criterion.} We wanted to extend those results to positive characteristic, which brings a few challenges. By far, the hardest challenge is the ability to prove closedness of an orbit. The points we need for our purposes are considerably complicated, for example:

\begin{problem} \label{prob:intro}
Let $V$ be a $3n$-dimensional vector space with basis $\{x_i,y_i,z_i\}_{1 \leq i \leq n}$. Consider the action of $\SL(V)$ on $W = \Sym^3(V)^{\oplus 2}$. Is $w = (\sum_{i=1}^n x_i^2 z_i, \sum_{i=1}^n y_i^2 z_i) \in W$ polystable? \end{problem}

In order to handle such cases, we develop a technique, also inspired by Kempf \cite{Kempf}, based on his theory of optimal 1-parameter subgroups. Our technique also has its limitations, for example, to be feasible, there needs to be significant symmetries for the point that is being investigated (which for example, the Dadok-Kac criterion does not require). In many situations, however, the points of interest often carry symmetries and moreover these symmetries are often the reason for their study. A high-level perspective of our approach can be summarized as follows -- search for optimal one-parameter subgroups (defined in Section~\ref{sec:kempf}) and if the search is unsuccessful, then the orbit is closed. The highly non-trivial part is to make the search for optimal 1-parameter subgroups feasible. A significant contribution of this paper is to develop the needed technical framework in order to utilize Kempf's theory to its full potential from a computational perspective. We succeed not just in our endeavor to extend exponential degree lower bounds for invariant rings to positive characteristic (we also improve the characteristic $0$ results), but also in exhibiting the usefulness of our technique in other contexts that are of interest to a wide mathematical audience, notably symmetric polynomials.

We now proceed to introducing the main results of this paper rigorously.


\section{Introduction and main results} \label{sec:intro}
First, we recall invariant theory and in particular, the notions of degree bounds, null cones and separating invariants. Next, we briefly explain the method to prove degree lower bounds for invariant rings via Grosshans principle and present our results on exponential degree lower bounds. Following that, we discuss our approach to proving polystability and the various results we are able to prove. 

\subsection{Invariant theory}
The subject of invariant theory has had a computational nature to its side from its very beginnings in the 19th century. The nature of computational results has evolved over the course of time in tandem with the mathematical community's understanding of the notions of computation and efficiency. In this century, driven by the Geometric Complexity Theory (GCT) program, computational invariant theory has evolved to incorporate notions of efficiency as described rigorously in the subject of computational complexity. Moreover, fundamental connections between the computational efficiency of invariant theoretic algorithms and central problems in theoretical computer science such as VP vs VNP (an algebraic analog of the celebrated P vs NP) and the polynomial identity testing problem have been discovered and has led to some important advances in recent times, starting with \cite{GCTV, FS, GGOW, DM, IQS2} and followed by more works such as \cite{DM-oc, AZGLOW, BFGOWW}.\footnote{See also \cite{GIMOWW, MW} for some recent negative results.}

The basic setup is as follows. Recall that our ground field $K$ is algebraically closed. Let $G$ be an algebraic group over $K$. Let $V$ be a rational representation of $G$, i.e., $V$ is a finite dimensional vector space, with a homomorphism of algebraic groups $\rho: G \rightarrow \GL(V)$. We write $g \cdot v$ or $gv$ for $\rho(g) v$. Let $K[V]$ denote the ring of polynomial functions on $V$ (a.k.a. the coordinate ring). Note that $K[V] = \Sym(V^*) = \oplus_{d = 1}^{\infty} \Sym^d(V^*)$ is the symmetric algebra over the dual $V^*$ and in particular a graded $K$-algebra. The orbit $O_v$ of a point $v$ is defined as $O_v := \{g v\ | \ g\in G\}$. A polynomial $f \in K[V]$ is called invariant if it is constant along orbits, i.e., $f(gv) = f(v)$ for all $g \in G, v\in V$. The collection of all invariant polynomials forms a graded subalgebra of $K[V]$ which we denote by $K[V]^G$ and call the {\em invariant ring} or {\em ring of invariants}. 

A group $G$ is called reductive if its unipotent radical is trivial. For a rational representation of a reductive group, the invariant ring is finitely generated, see \cite{Hilb1, Hilb2, Nagata, Haboush}. A central question in computational invariant theory is to efficiently describe a set of generators (as a $K$-algebra) for the ring of invariants $K[V]^G$. The problem of {\em degree bounds} is often a first step:

\begin{problem} [Degree bounds]
For a rational representation $V$ of a reductive group $G$, find strong bounds for the maximal degree of a set of (minimal) generators for $K[V]^G$, i.e, bounds for
$$
\beta(G,V) = \min \{d\ |\ K[V]^G_{\leq d} \text{ is a system of generators}\}.
$$
\end{problem}

Degree bounds has been studied for several decades, see \cite{Popov, Popov2, Derksen} and references therein. Nevertheless, the aforementioned connections to complexity has given the problem a new significance. For example, {\em polynomial} degree bounds for matrix semi-invariants \cite{DM}\footnote{see \cite{DM-siq} for extensions to quivers and \cite{DM-arbchar} for extension to positive characteristic. For applications of these results, see e.g., \cite{Domokos, GGOW-BL, KPV, LQWD, MLE-MNM, MLE-TNM} and references therein.} were crucial in obtaining a polynomial time (algebraic) algorithm for the problem of non-commutative rational identity testing (RIT) \cite{IQS2}.\footnote{A polynomial time analytic algorithm for RIT precedes the algebraic algorithm and does not use degree bounds \cite{GGOW}. However, the analytic algorithm does not have an analog in positive characteristic, whereas the algebraic algorithm works in all characteristics.}

The zero set of a collection of polynomials $S\subseteq K[V]$ is
$$
\VV(S)=\{v\in V\mid f(v)=0\mbox{ for all $f\in S$}\}.
$$
Hilbert's null cone $\nullcone\subseteq V$ is defined by $\nullcone=\VV(\bigoplus_{d=1}^\infty K[V]_d^G)$.
\begin{definition}
We  define $\sigma(G,V)$ to be the smallest integer $D$ such that the non-constant homogeneous invariants of degree $\leq D$ define the null cone,
so 
$$
\textstyle\sigma(G,V)=\min\Big\{D\,\Big|\, \nullcone=\VV\big(\bigoplus_{d=1}^D K[V]^G_d\big)\Big\}.
$$
\end{definition}
General upper bounds for $\sigma_G(V)$ were first given by Popov (see~\cite{Popov,Popov2}), and improved by the first author in \cite{Derksen}. For any system of generating invariants, its zero locus is the null cone, so it is clear that
$$
\beta(G,V) \geq \sigma(G,V).
$$
In characteristic zero, the first author showed that $\beta(G,V)$ and $\sigma(G,V)$ are polynomially related \cite{Derksen}. A central problem in algorithmic invariant theory is the {\em orbit closure intersection} problem -- given $v,w \in V$, decide if $\overline{O_v} \cap \overline{O_w} = \emptyset$. This problem in various instances captures many important problems in mathematics, computer science and physics, see e.g., \cite{BFGOWW} and references therein. The following result is due to Mumford and captures why invariant polynomials are useful for the problem of orbit closure intersection.

\begin{theorem} \label{thm:Mumford}
Let $V$ be a rational representation of a reductive group $G$. Let $v,w \in V$. Then
$$
\overline{O_v} \cap \overline{O_w} \neq \emptyset \Leftrightarrow f(v) = f(w)\ \forall\ f \in K[V]^G.
$$
\end{theorem}

Clearly, a system of generating invariants are sufficient for detecting orbit closure intersection. However, this approach is rarely efficient due to familiar complexity theoretic barriers \cite{GIMOWW}. Yet, one can often get away with a smaller set of invariants. A subset $S \subseteq K[V]^G$ is called a {\em separating subset} if for every $v,w \in V$ such that $\overline{O_v} \cap \overline{O_w} = \emptyset$, there exists $f \in S$ such that $f(v) \neq f(w)$.

\begin{definition}
We define $\beta_{\rm sep}(G,V)$ to be the smallest integer $D$ such that the invariants of degree $\leq D$ form a separating subset.
\end{definition}

Clearly, we have:
\begin{equation} \label{eq:deg.bds.ineq}
\beta(G,V) \geq \beta_{\rm sep}(G,V) \geq \sigma(G,V).
\end{equation}

Separating subsets can be much better behaved than generating subsets in positive characteristic, see e.g., \cite{DKW}. One concrete instance in which $\beta(G,V)$ has been proven to be strictly larger than $\beta_{\rm sep}(G,V)$ is the case of matrix invariants \cite{DM-oc}.

We end this subsection by recalling the notions of stability:
\begin{definition}
Let $\rho: G \rightarrow \GL(V)$ be a rational representation of a reductive group. Let $v \in V$. We say $v$ is:
\begin{itemize}
\item unstable, if $0 \in \overline{O_v}$;
\item semistable, if $0 \notin \overline{O_v}$;
\item polystable, if $v \neq 0$ and $O_v$ is closed;
\item stable, if $v$ is polystable and $\dim(G_v) = \dim(\text{kernel of } \rho)$.

\end{itemize}
\end{definition}

Note in particular that Theorem~\ref{thm:Mumford} implies that the null cone is precisely the subset of unstable points.

\subsection{Grosshans principle and exponential degree lower bounds}
Constructing torus actions with exponential degree bounds is an excursion in linear algebra, see e.g., \cite[Section~3]{DM-exp}. Indeed, the invariant theory for torus actions is much better understood (see \cite{Wehlau}). Tori happen to be commutative reductive groups and in fact any connected commutative reductive group is a torus. The invariant theory for non-commutative groups is much harder and in general it is difficult to even write down invariants \cite{GIMOWW}. In characteristic zero, we gave a surjection from the invariant rings for cubic forms and tensor actions to the invariant ring for a torus action. This allows one to ``lift" lower bounds on invariant rings for tori to lower bounds on invariant rings for cubic forms and tensor actions. Such a result is known as a {\em lifting} theorem in complexity theory. In numerous areas of complexity, various {\em lifting} techniques and barriers to them have been studied, see e.g., \cite{Raz90, NW96, BPR97, RM99, EGOW, GMOW}.

In positive characteristic, the theory breaks down in a predictable way because of the existence of non-smooth reductive groups. Nevertheless, we are able to lift bounds for separating invariants, which we will explain below. First, we state Grosshans principle \cite{Grosshans}.
We let the group $H\times G$ act on $G$ by $(h,g)\cdot u=hug^{-1}$.

\begin{theorem} [Grosshans principle] \label{gross.princ}
Let $W$ be a representation of $G$, and let $H$ be a closed subgroup of $G$. Then we have an isomorphism
$$
\psi:(K[G]^H \otimes K [W])^G \longrightarrow K[W]^H.
$$
\end{theorem}

From Grosshans principle, we will derive the following main technical result:

\begin{theorem} \label{thm:liftingbounds}
Let $V,W$ be rational representations of a reductive group $G$. Suppose $v \in V$ is such that its $G$-orbit is closed and let $H = G_v = \{g \in G \ |\ gv = v \}$. Then
$$
\beta(G,V \oplus W) \geq \beta_{\rm sep}(G,V \oplus W) \geq \beta_{\rm sep}(H,W) \geq \sigma(H,W).
$$

\end{theorem}

It is clear that to use this method in any meaningful way, one must be able to prove that an orbit is closed, which we discuss in the next subsection. For now, we state our results on exponential degree lower bounds. First, our result on cubic forms:

\begin{theorem} \label{thm:cubicbounds}
Assume $\kar(K) \neq 2$. Let $V$ be a $3n$-dimensional vector space, and consider the natural action of $\SL(V)$ on $\Sym^3(V)^{\oplus 3}$, the space of triplets of cubic forms. Then, 
$$
\beta(\SL(V), \Sym^3(V)^{\oplus 3}) \geq {\textstyle\frac{2}{3}} (4^n-1).
$$
\end{theorem}

Our result on tensor actions: 

\begin{theorem} \label{thm:tensorbounds}
Let $U,V,W$ be $3n$-dimensional vector spaces. Consider the natural action of $G = \SL(U) \times \SL(V) \times \SL(W)$ on $(U \otimes V \otimes W)^{\oplus 5}$. Then
$$
\beta(G, (U \otimes V \otimes W)^{\oplus 5}) \geq 4^n-1
$$
\end{theorem}

The importance of the above two results is best understood in the context of GCT as has already been explained in \cite{DM-exp}.

\subsection{Closed orbits}
The following result captures essentially our strategy for proving closed orbits.

\begin{theorem}\label{thm:intro-closed-orbits}
Let $V$ be a rational representation of a reductive group $G$. Let $v \in V$ and let $G_v := \{g \in G\ |\ gv = v\}$ be its stabilizer. Let $H \subseteq G_v$ be a maximal torus of $G_v$. Let $\mathcal{T}$ be a subset of maximal tori of $G$ such that
\begin{enumerate}
\item For every parabolic $P \supseteq G_v$, there exists $T \in \mathcal{T}$ such that $T \subseteq P$;
\item For every $T \in \mathcal{T}$, there is an inclusion $kHk^{-1} \subseteq T$ for some $k \in G_v$ (that can depend on $T$).
\end{enumerate}
Then, the orbit $O_v$ is closed if and only if $T \cdot v$ is closed for all $T \in \mathcal{T}$
\end{theorem}

We state a few variants/generalizations of the above theorem in Section~\ref{sec:kempf} and picking the right variant/generalization can often make things much easier. In particular, we will mildly strengthen some of Kempf's statements. 

Let us briefly summarize what is required to be able to use the above theorem effectively. For any given torus $T$, checking whether the torus orbit $T \cdot v$ is closed is not so difficult, see Section~\ref{sec:torus-actions}\footnote{There is even a polynomial time algorithm for this \cite{BDMWW}.}. 
 Finding a collection of maximal tori that satisfy the first condition while at the same time having computational feasibility is much harder. For example, if $G_v$ is trivial, then one has to take $\mathcal{T}$ to be all maximal tori of $G$, which is computationally infeasible. A $G_v$ that severely restricts the potential parabolic groups containing it is needed to make the computation tractable.


In the case when $G = \SL(V)$ (or a product of $\SL$'s), we can be a bit more explicit. Parabolic subgroups can be seen as subgroups that fix a flag of subspaces in $V$. If a parabolic subgroup contains $G_v$, then the corresponding flag consists of $G_v$-stable subspaces of $V$. Hence, in the cases where there are very few $G_v$-subrepresentations of $V$, the technique is particularly useful. Perhaps more interestingly, even in some cases where we have an infinite number of $G_v$-stable subspaces, the technique can still be applied successfully and this is actually needed for our results on exponential degree bounds!


\subsection{New results on polystability for polynomials}
In this paper, we take a more detailed look at polynomials, particularly those with symmetries. We prove several results with respect to polystability (and semistability, unstability, etc), some of which are new even in characteristic zero. Consider the defining action of the special linear group $\SL_n(K)$ on $K^n$. Let $x_1,\dots,x_n$ denote the standard basis for $K^n$, and consider the natural induced action of $\SL_n(K)$ on $\Sym^d(K^n) = K[x_1,\dots,x_n]_d$, the space of degree $d$ polynomials in $x_1,\dots,x_n$. 

\begin{definition}
We say $f \in K[x_1,\dots,x_n]_d$ is unstable (resp. semistable, polystable, stable)  if it is $\SL_n(K)$-unstable (resp. semistable, polystable, stable). 
\end{definition}

We say an exponent vector $e = (e_1,\dots,e_n)$ is {\em entirely even} if all the $e_i$s are even. For a polynomial $f = \sum_{e} c_e x^e \in K[x_1,\dots,x_n]$, we define the support $\supp(f) = \{e \in \N^n \ |\ c_e \neq 0\} \subseteq \N^n \subseteq \Q^n \subseteq \R^n$. The Newton polytope of a polynomial $f$ is the convex hull of its support and is denoted $\NP(f)$.

\begin{theorem} \label{thm:evenpoly}
Let $\kar(K) \neq 2$. Let $f = \sum_{e \text{ entirely even}} c_e x^e \in K[x_1,\dots,x_n]_d$. Then $f$ is
\begin{itemize}
\item semistable if and only if $(\frac{d}{n}, \frac{d}{n}, \dots,\frac{d}{n}) \in \NP(f)$;
\item polystable if and only if $(\frac{d}{n}, \frac{d}{n}, \dots,\frac{d}{n})$ is in the relative interior of $\NP(f)$;
\end{itemize}
\end{theorem}

\begin{remark}
The above result also works if we replace entirely even exponent vectors with entirely $0$ mod $d$ exponent vectors for any $d > 2$ (as long as $p \nmid d$). Also, observe that in characteristic $0$, such a result follows easily from the Dadok-Kac criterion \cite{Dadok-Kac}. The argument we use in positive characteristic is far more subtle.
\end{remark}

We now turn to symmetric polynomials. A polynomial $f \in K[x_1,\dots,x_n]$ is called symmetric if it is invariant under permutations of the $x_i$'s, i.e., $f \in K[x_1,\dots,x_n]^{S_n}$. Symmetric polynomials have been intensely studied for over a century with diverse motivations and serve to interconnect many disparate fields. We refer the interested reader to Macdonald's seminal text \cite{Macdonald}.  Yet, there seems to have been relatively little work on polystability. 

It turns out that Theorem~\ref{thm:intro-closed-orbits} or its variants are not quite sufficient for our purposes and we have to additionally leverage the relationship between optimal $1$-parameter subgroups and optimal parabolic subgroups. For example, let $p \nmid n$ with no restriction on $d$. Then, to decide polystability of a symmetric polynomial $f$ of degree $d$ in $n$ variables, we show that one only needs to understand the limits (at $0$ and $\infty$) of precisely one $1$-parameter subgroup, see Lemma~\ref{pnmidn-1-onepsg} for a precise statement. In particular, such results enable us to give an algorithm to decide polystability of symmetric polynomials. 

\begin{theorem} \label{thm:symm-algo}
Let $f$ be a homogeneous symmetric polynomial of degree $d$ in $n$ variables. Then Algorithms~\ref{algo:pnmidn} and~\ref{algo:pmidn} can decide if $f$ is unstable (resp. semistable, polystable, stable). 
\end{theorem}
We refrain from giving a complexity-theoretic analysis of Algorithms~\ref{algo:pnmidn} and~\ref{algo:pmidn} as it digresses too far from the scope of this paper. However, the complexity of most of the individual steps in the algorithm are well known, and perhaps the non-trivial part is to establish what is the right way to input a symmetric polynomial, etc.

Independent of the complexity of the algorithms, we are also able to prove a number of results on polystability of symmetric polynomials. We state only here to avoid introducing too much notation in the introduction, see Section~\ref{sec:symm-polys} for more such results. For a partition $\lambda \vdash d$, let $s_\lambda(x_1,\dots,x_n)$ denote the Schur polynomial associated to $\lambda$ in $n$ variables (see Section~\ref{sec:symm-polys} for the definition).

\begin{theorem} \label{thm:Schur}
Let $\kar(K) = 0, d \geq 2$ and $\lambda \vdash d$. Then for any $n > d$, the Schur polynomial $s_\lambda(x_1,\dots,x_n)$ is polystable.
\end{theorem}

\subsection{Organization}
In section~\ref{sec:torus-actions}, we recall the computational invariant theory for torus actions. Section~\ref{sec:kempf} is devoted to discussing Kempf's theory of optimal subgroups and the consequences of it for the purposes of determining polystability. We review the representation theory of the special linear group and focus on the computational aspects relevant for us and prove Theorem~\ref{thm:evenpoly} in section~\ref{sec:rep-theory-background}. In section~\ref{sec:closed-orbits-degree-bounds}, we prove that orbits of certain points (which are relevant for degree lower bounds) are closed. In section~\ref{sec:Grosshans}, we explain our technique using Grosshans principle, i.e., Theorem~\ref{thm:liftingbounds} and prove exponential degree lower bounds for cubic forms and tensor actions, i.e., Theorem~\ref{thm:cubicbounds} and Theorem~\ref{thm:tensorbounds}. Section~\ref{sec:symm-polys} discusses polystability of symmetric polynomials and in particular gives an algorithm for it, i.e, we prove Theorem~\ref{thm:symm-algo}. Finally, in section~\ref{sec:symm-poly-class}, we determine the polystability of certain classes of interesting symmetric polynomials.

\section{Invariant theory for torus actions} \label{sec:torus-actions}
The invariant theory for torus actions is well studied (see \cite{Wehlau} or \cite[Section~3]{DM-exp}). We will briefly recall the important statements:

 Let $T = (K^*)^m$ be an $m$-dimensional torus. Let $\mathcal{X}(T)$ denote the set of characters or weights (i.e., morphisms of algebraic groups $T \rightarrow K^*$). For each $\lambda \in \Z^m$, we associate a character, also denoted $\lambda$ by abuse of notation, defined by the formula $\lambda(t_1,\dots,t_m) = \prod_{i=1}^m t_i^{\lambda_i}$. This defines an isomorphism of abelian groups from $\Z^m$ to $\mathcal{X}(T)$. Now, suppose $V$ is a representation of $T$. A vector $v \in V$ is called a weight vector of weight $\lambda \in \Z^m = \mathcal{X}(T)$ if $t \cdot v = \lambda(t) v$ for all $t \in T$. We have a weight space decomposition
$$
V = \oplus_{\lambda \in \Z^m} V_\lambda,
$$
where $V_\lambda = \{v \in V \ |\ t \cdot v = \lambda(t) v\}$. In particular, we have a basis consisting of weight vectors. 

Let $\mathcal{E} = (e_1,\dots,e_n)$ be a basis of $V$ consisting of weight vectors. Suppose the weight of $e_i$ is $\lambda^{(i)}$. Using this basis, identify $V$ with $K^n$, which then allows us to identify $K[V]$ with the polynomial ring $K[z_1,\dots,z_n]$. A monomial $z_1^{c_1} z_2^{c_2} \dots z_n^{c_n} \in K[V]$ is invariant if and only if $\sum_i c_i \lambda^{(i)} = 0$. Moreover, the invariant ring $K[V]^T = K[z_1,\dots,z_n]^T$ is linearly spanned by such invariant monomials. We refer to \cite[Section~3]{DM-exp} for more details. For a vector $v$, consider its support ${\rm Supp}(v) = \{i \ |\ v_i \neq 0\}$. Then, we define its {\em weight polytope} $\WP(v)$ to be the convex hull of the points $\{\lambda^{(i)}\ |\ i \in {\rm Supp}(v)\}$ thought of as a subset of $\R^m$. Even without coordinates with respect to an explicit basis, one can define the weight polytope. For each $v \in V$, we can write $v = \sum_\lambda v_\lambda$ where $v_\lambda \in V_\lambda$. Then, the weight polytope $\WP(v)$ is the convex hull of the points $\{\lambda\ |\ v_\lambda \neq 0\}$. We call $\{\lambda\ |\ v_\lambda \neq 0\}$ the weight set of $v$.

\begin{lemma}
Let $V$ be an $n$-dimensional representation of an $m$-dimensional torus $T = (K^*)^m$, and let $\mathcal{E} = (e_1,\dots,e_n)$ be a weight basis such that the weight of $e_i$ is $\lambda^{(i)}$. Let $0 \neq v \in V$. Then,
\begin{itemize}
\item $v$ is semistable if and only if $0 \in \WP(v)$;
\item $v$ is polystable if and only if $0$ is in the relative interior of $\WP(v)$;
\item $v$ is stable if and only if $0$ is in the interior of $\WP(v)$.
\end{itemize}
\end{lemma}

For concreteness, we discuss the action of $\ST_n$ the group of diagonal $n \times n$ matrices with determinant $1$ on $\Sym^d(K^n)$, the space of degree $d$ polynomials in $x_1,\dots,x_n$. For a polynomial $f \in \Sym^d(K^n)$, write $f = \sum_{e \in \N^n} c_e x^e$. We define the Newton polytope
$$
\NP(f) = \text{convex hull } \{e \in \N^n \ | c_e \neq 0\}.
$$

We think of $\NP(f)$ not as a subset of $\R^n$, but as a subset of $E_d = \{(v_1,\dots,v_n) \in \R^n \ |\ \sum_i v_i = d\}$. This is necessary for the last part of the following corollary. In this special case, the above lemma translates to the following:

\begin{corollary} \label{cor:symm-torus}
Consider the action of $\ST_n$ on $\Sym^d(K^n)$. Let $ 0 \neq f \in \Sym^d(K^n)$. Then
\begin{itemize}
\item $f$ is semistable if and only if $(\frac{d}{n},\frac{d}{n},\dots,\frac{d}{n}) \in \NP(f)$;
\item $f$ is polystable if and only if $(\frac{d}{n},\frac{d}{n},\dots,\frac{d}{n})$ is in the relative interior of $\NP(f)$;
\item $f$ is stable if and only if $(\frac{d}{n},\frac{d}{n},\dots,\frac{d}{n})$ is in the interior of $\NP(f)$.
\end{itemize}
\end{corollary}

For torus actions, semistability, polystability and stability are all determined by the weight polytopes. For polystability we draw a connection to invariant monomials.

\begin{lemma} \label{lem:polystable-wp}
Let $V$ be an $n$-dimensional representation of an $m$-dimensional torus $T = (K^*)^m$, and let $\mathcal{E} = (e_1,\dots,e_n)$ be a weight basis such that the weight of $e_i$ is $\lambda^{(i)}$. Let $0 \neq v \in V$. Then, the following are equivalent:
\begin{itemize}
\item $v$ is polystable.
\item For every $i \in {\rm Supp}(v)$, there exists an invariant monomial $\prod\limits_{j \in {\rm Supp}(v)} x_j^{c_j}$ such that $c_i \geq 1$.
\end{itemize}
\end{lemma}

We recall the Hilbert--Mumford criterion, for which we need to understand $1$-parameter subgroups. A $1$-parameter subgroup of $T$ is a morphism of algebraic groups $\lambda: K^* \rightarrow T$. We denote by $\Gamma(T)$, the set of $1$-parameter subgroups of $T$. If we identity $T$ with $(K^*)^m$, then any $1$-parameter subgroup is of the form $t \mapsto (t^{a_1}, t^{a_2},\dots,t^{a_m})$ where $a_i \in \Z$. There is an abelian group structure on $\Gamma(T)$, if we take $1$-parameter subgroups $ t \mapsto (t^{a_1}, t^{a_2},\dots,t^{a_m})$ and $t \mapsto (t^{b_1}, t^{b_2},\dots,t^{b_m})$, we can multiply them to get another $1$-parameter subgroup $t \mapsto (t^{a_1+b_1}, t^{a_2+b_2},\dots,t^{a_m+b_m})$. This allows us to also identify $\Gamma(T)$ with $\Z^m$ as an abelian group.

\begin{theorem} [Hilbert--Mumford criterion]
Let $V$ be an $n$-dimensional representation of an $m$-dimensional torus $T = (K^*)^m$, and let $\mathcal{E} = (e_1,\dots,e_n)$ be a weight basis such that the weight of $e^{(i)}$ is $\lambda^{(i)}$. Let $v \in V$ and consider its orbit $O_v$. Let $S$ be another closed $T$-stable subset of $V$. Then $S \cap \overline{O}_v \neq \emptyset$ if and only if there exists a $1$-parameter subgroup $\lambda$ such that $\lim_{t \to 0} \lambda(t) \cdot v \in S$.
\end{theorem}

\begin{definition}
Let $V$ be an $n$-dimensional representation of an $m$-dimensional torus $T = (K^*)^m$, and let $\mathcal{E} = (e_1,\dots,e_n)$ be a weight basis such that the weight of $e_i$ is $\lambda^{(i)}$. Let $v = (v_1,\dots,v_n) \in V$, where $v_i$ are coordinates in the basis $\mathcal{E}$. We say $i \in {\rm Supp}(v)$ is {\em essential} if there exists a non-negative linear combination $\sum_{j \in {\rm Supp}(v)} c_j \lambda^{(j)} = 0$ with $c_i > 0$. We define
$$
{\rm eSupp}(v) = \{i \in [n] \ | \ i \text{ is essential}\},
$$
and we define ${\rm ess}(v) = v|_{{\rm eSupp}(v)}$ by 
$$
{\rm ess}(v)_i = \begin{cases} v_i & \text{ if $i \in {\rm eSupp}(v)$} \\ 0 & \text{ otherwise.} \end{cases}.
$$
\end{definition}

The following lemma is well-known, see e.g. \cite[Example~1.3]{Popov-occ}

\begin{lemma}
Let $V$ be an $n$-dimensional representation of an $m$-dimensional torus $T = (K^*)^m$, and let $\mathcal{E} = (e_1,\dots,e_n)$ be a weight basis such that the weight of $e_i$ is $\lambda^{(i)}$. Let $v \in V$. Then ${\rm ess}(v)$ is a point in the unique closed $T$-orbit inside $\overline{O_{T,v}}$.
\end{lemma}

\section{Kempf's theory of optimal subgroups} \label{sec:kempf}
Let $G$ be a reductive algebraic group over $K$. Let $V$ be a rational representation of $G$. We will recall some technical notions that we need to be able to state the known results on optimal subgroups in a coherent fashion. We are content to only briefly recall these notions, referring the interested reader to \cite{Kempf} for more details. Much of these technical notions are only required to prove our main results in this section, but are not needed in their statements or anywhere else in this paper.

A $1$-parameter subgroup of $G$ is a morphism of algebraic groups $\lambda: K^* \rightarrow G$. Let $\Gamma(G)$ denote the set of $1$-parameter subgroups of $G$.  A length function $||\ ||$ on $\Gamma(G)$ is a non-negative real valued function such that 
\begin{itemize}
\item $||g \cdot \lambda|| = ||\lambda||$ for any $g \in G$ and $\lambda \in \Gamma(G)$;
\item For any maximal torus $T$ of $G$, there is a positive-definite integral valued bilinear form $(\ ,\ )$ on $\Gamma(T)$ such that $(\lambda,\lambda) = ||\lambda||^2$ for any $\lambda \in \Gamma(T)$.
\end{itemize}

Recall that $\Gamma(T)$ represents the set of $1$-parameter subgroups of $T$. It is a free abelian group of rank equal to the dimension of $T$. The existence of a length function is not obvious but not very involved either. Pick a maximal torus $T$, pick a positive-definite integral valued bilinear form $(\ ,\ )$ on $\Gamma(T)$ which is invariant under the Weyl group. Thus, for any $\lambda \in \Gamma(G)$, we view it in $\Gamma(gTg^{-1})$ for some $g \in G$ (since all maximal tori are conjugate and any $1$-parameter subgroup lies in a maximal torus) and we define $|| \lambda|| = (g^{-1}\lambda g ,g^{-1}\lambda g)$. Note that $g^{-1} \lambda g \in \Gamma(T)$. That such a length function is well-defined is a consequence of the invariance of the bilinear form under the Weyl group.

Let $v \in V$. Let $S \subseteq V$ be a closed $G$-subscheme. Let $|S,v|$ denote the set of all $1$-parameter subgroups $\lambda$ of $G$ such that the $\lim_{t \to 0} \lambda(t) \cdot v$ exists in $S$. The set $|S,v|$ is non-empty if and only if $S \cap \overline{O_v} \neq \emptyset$, see \cite[Theorem~1.4]{Kempf}. For $\lambda \in |S,v|$, let $M(\lambda): \mathbb{A}^1 \rightarrow V$ be the unique morphism defined by $M(\lambda)(t) = \lambda(t) \cdot v$ for $t \neq 0$ (where $\mathbb{A}^1$ denote the $1$-dimensional affine space). Let $S$ denote a subscheme of $V$ that is closed under the action of $G$. Then, let $a_{S,v}(\lambda)$ denote the degree of the divisor $M(\lambda)^{-1} S$ (which is an effective divisor on $\mathbb{A}^1$). If $|S,v| \neq \emptyset$, the function $\lambda \mapsto a_{S,v}(\lambda)/||\lambda||$ takes a maximum value $B_{S,v}$ on $|S,v|$. 

A 1-parameter subgroup $\lambda$ is called divisible if there exists another 1-parameter subgroup $\mu$ and a postive integer $r \geq 2$ such that $\lambda(t) = \mu(t)^r$ for all $t \in K^*$. A 1-parameter subgroup that is not divisible is called indivisible. When $|S,v| \neq \emptyset$, an indivisible 1-parameter subgroup $\lambda \in |S,v|$ is called optimal if $a_{S,v}(\lambda)/||\lambda|| = B_{S,v}$. We denote by $\Lambda(S,v)$ the set of optimal 1-parameter subgroups.

For a 1-parameter subgroup $\lambda$, we define the associated parabolic subgroup $P(\lambda) = \{g \in G \ | \lim_{t \to 0} \lambda(t) g \lambda(t^{-1}) \in G\}$.

We summarize the main technical results from \cite[Section~3]{Kempf}).

\begin{theorem} \label{thm:Kempf}
Let $V$ be a rational representation of a reductive group $G$. Let $v \in V$ such that $O_v$ is not closed. Let $S$ be a closed $G$-stable subscheme such that $O_v \cap S = \emptyset$ and $\overline{O_v} \cap S \neq \emptyset$. Fix a choice of length function $|| \ ||$ on $\Gamma(G)$. Let $G_v$ denote the stabilizer of $v$.
\begin{enumerate}
\item The set $\Lambda(S,v)$ of optimal 1-parameter subgroups is non-empty.
\item There is a parabolic subgroup $P_{S.v}$ such that $P(\lambda) = P_{S,v}$ for all $\lambda \in \Lambda(S,v)$. We call $P_{S,v}$ the optimal parabolic subgroup;
\item Any maximal torus of $P_{S,v}$ contains a unique member of $\Lambda(S,v)$;
\item $G_v \subseteq P_{S,v}$.

\end{enumerate}
\end{theorem}

\subsection{Results on polystability}
Using Theorem~\ref{thm:Kempf}, we give a proof of Theorem~\ref{thm:intro-closed-orbits}.

\begin{proof} [Proof of Theorem~\ref{thm:intro-closed-orbits}]
Suppose the orbit $O_v$ is closed. Then, we claim that for any torus $T \supseteq H$, the $T$-orbit $O_{T,v}$ is closed. To see this, consider the action of $T$ on $O_v$. For any $g \in G$, the $T$-stabilizer at $w = gv \in O_v$ is given by $T \cap g G_v g^{-1}$. Thus, the $\dim(T_w) \leq $ rank of any maximal torus in $gG_vg^{-1}$ = rank of $H$. Thus, $\dim(T \cdot v) \leq \dim(T \cdot w)$ for any $w \in O_v$. For the action of any reductive group on a variety, an orbit of the smallest possible dimension must always be closed (since the boundary of an orbit, if non-trivial, contains orbits of smaller dimension). In particular, for the action of $T$ on $O_v$, this means that $T \cdot v$ is closed.

Conversely, suppose $O_v$ is not closed. Then let $S = \overline{O_v} \setminus O_v$. Then, let $P := P_{S,v}$ be the optimal parabolic subgroup as in Theorem~\ref{thm:Kempf}. Since $P \supseteq G_v$ by Theorem~\ref{thm:Kempf}, there exists $T \in \mathcal{T}$ such that $T$ is a maximal torus of $P$. Further, Theorem~\ref{thm:Kempf} says that there is an optimal $1$-parameter subgroup in $T$, which in the limit drives $v$ out of its $G$-orbit (and hence out of its $T$-orbit). Thus, the $T$-orbit of $v$ is not closed for this particular $T$.
\end{proof}


To use Theorem~\ref{thm:intro-closed-orbits}, one must be able to compute $G_v$ or at least a maximal torus of it. Such a computation may not always be possible. In that case, one can use:
\begin{theorem} \label{thm:main2}
Let $V$ be a rational representation of a reductive group $G$. Let $v \in V$ and let $G_v = \{g \in G \ |\  gv = v\}$. Let $\mathcal{T}$ be a subset of maximal tori of $G$ such that for every parabolic $P \supseteq G_v$, there exists $T \in \mathcal{T}$ such that $T \subseteq P$.
Then, 
$$
O_{G,v} \text{ is closed } \Longleftrightarrow \overline{O_{T,v}} \subseteq O_{G,v}\  \forall T \in \mathcal{T}.
$$
\end{theorem}

\begin{proof}
If $O_{G,v}$ is closed, then clearly $\overline{O_{T,v}} \subseteq O_{G,v}\  \forall T \in \mathcal{T}$. Now, suppose $O_{G,v}$ is not closed. Then, consider $P_{S,v}$ where $S = \overline{O_v} \setminus O_v$. Then $P_{S,v} \supseteq G_v$ by part $(4)$ of Theorem~\ref{thm:Kempf}. Thus, there exists $T \in \mathcal{T}$ such that $T \subseteq P_{S,v}$. Hence, by part~(3) of Theorem~\ref{thm:Kempf}, there must be a $1$-parameter subgroup $\lambda$ of $T$ which is optimal, so $\lim_{t \to 0} \lambda(t) v = w \notin O_{G,v}$. Hence $w \in \overline{O_{T,v}} \setminus O_{G,v}$ as required.
\end{proof}

The obvious issue here is that we have to be able to tell when the closure of $O_{T,v}$ is contained in $O_{G,v}$. It would be simplest if $O_{T,v}$ is itself closed. If that is not the case, then the following is one way to test whether $\overline{O_{T,v}} \subseteq O_{G,v}$

\begin{lemma}
Let $V$ be a representation of $G$ and let $T$ be a maximal torus. Let $\mathcal{E} = (e_1,\dots,e_n)$ be a weight basis for the action of $T$ on $V$ such that $T$ acts on $e_i$ by a weight $\lambda^{(i)}$. Then, let $v = (v_1,\dots,v_n)$ be the coordinates of $V$ in the basis $\mathcal{E}$. Let $w = {\rm ess}(v)$. Then,
$$
\overline{O_{T,v}} \subseteq O_{G,v} \Longleftrightarrow \dim(G_w) = \dim(G_v).
$$
\end{lemma}

\begin{proof}
Suppose $\overline{O_{T,v}} \subseteq O_{G,v}$, then clearly $w \in O_{G,v}$, so $\dim(G_w) = \dim(G_v)$ since $G_w$ and $G_v$ are conjugate subgroups. On the other hand, suppose $\overline{O_{T,v}} \nsubseteq O_{G,v}$. Let $Y = \overline{O_{G,v}}$. Now $A = \overline{O_{T,v}}$ is a Zariski-closed subset of $Y$ and $B = O_{G,v}$ is a Zariski-open subset of $Y$ and its complement $B^c = Y \setminus B$ is a Zariski-closed subset. Thus $A \cap B^{c}$ is a Zariski-closed subset of $Y$ which is $T$-stable (i.e., a union of $T$-orbits). In particular, this means that $w = {\rm ess}(v) \in A \cap B^c$ since any Zariski-closed $T$-stable subset of $\overline{O_{T,v}}$ contains $w$. Since $w \in  \overline{O_{G,v}} \setminus O_{G,v}$, we get that $\dim(G_w) > \dim(G_v)$.

\end{proof}
  
It is another matter that it may be quite hard to compute $G_w$ or $G_v$ completely. Yet, one can decide if $\dim(G_w) = \dim(G_v)$. In characteristic $0$, a Lie algebra computation will suffice and in characteristic $p >0$, one can use Gr\"obner basis techniques, see for example \cite[Chapter~9]{Cox-Little-Oshea}.

\subsection{Results on semistability}
To detect polystability, we used Theorem~\ref{thm:Kempf} with $S = \overline{O_v} \setminus O_v$. To detect semistability, one has to take $S = \{ 0 \}$ instead. Unlike the case of polystability, there is no need for variants, we state the most general version possible. The proofs are very similar to the previous subsection, so we leave the details to the reader.

\begin{theorem} \label{thm:main.nullcone}
Let $V$ be a rational representation of a reductive group $G$. Let $v \in V$ and define $G_v := \{g \in G\ |\ gv = v\}$. Let $\mathcal{T}$ be a subset of maximal tori of $G$ such that for every parabolic $P \supseteq G_v$, there exists $T \in \mathcal{T}$ such that $T \subseteq P$.
Then we have
$$
\text{$v$ is $G$-semistable} \Longleftrightarrow \text{ $v$ is $T$-semistable} \ \mbox{ for all } T \in \mathcal{T}.
$$
\end{theorem}

\subsection{Improvements}
In the case when the action of $G$ on $V$ extends to a larger group $\widetilde{G}$ containing $G$ as a normal subgroup, we can make certain improvements to Theorem~\ref{thm:main2} and Theorem~\ref{thm:main.nullcone}. These improvements are very handy in computations, especially in the case where $G$ is the special linear group and $\widetilde{G}$ is the general linear group.

First, we prove the following result that generalizes part $(4)$ of Theorem~\ref{thm:Kempf}.

\begin{proposition} \label{prop:improve-kempf}
Let $\widetilde{G}$ be an algebraic group and let $G$ be a reductive normal subgroup. Let $V$ be a rational representation of $\widetilde{G}$ and hence of $G$ as well. Suppose that  $v \in V$ and let $S=O_{G,\widehat{v}}$ be the unique closed $G$-orbit in $\overline{O}_{G,v}$. Fix a choice of length function $|| \ ||$ on $\Gamma(G)$ and let $P_{S,v}$ be the optimal parabolic subgroup. Then
$$
h P_{S,v} h^{-1} = P_{S,v} \text{ for all } h \in \widetilde{G}_v.
$$
\end{proposition}

\begin{proof}
For $h \in \widetilde{G}$ and $\lambda \in \Gamma(G)$, we define $h \ast \lambda$ by $h \ast \lambda (t) = h \lambda(t) h^{-1}$ which is a $1$-parameter subgroup of $G$ because $G$ is normal in $\widetilde{G}$. Clearly, $|hS, hv| = h \ast |S,v|$. 

Now, suppose $h \in \widetilde{G}_v$. Then we get $|hS,v| = |hS, hv| = h \ast |S,v| \neq \emptyset$. 
Moreover, observe that $hS = h O_{G,\widehat{v}} = O_{G, h \widehat{v}}$ since $G$ is normal in $\widetilde{G}$. In particular, this means that $hS$ is a closed $G$-orbit. Since $|hS,v|$ is not empty, $hS$ must be the unique closed orbit in $\overline{O}_{G,v}$, so $hS = S$. This means $|S,v| = h \ast |S,v|$ which implies immediately that $h P_{S,v} h^{-1} = P_{S,v}$. 
\end{proof}

By using the above proposition instead of part $(4)$ of Theorem~\ref{thm:Kempf}, we get the following improvement of Theorem~\ref{thm:main2} and Theorem~\ref{thm:main.nullcone}:

\begin{theorem} \label{thm:main3}
Let $\widetilde{G}$ be an algebraic group and let $G$ be a reductive normal subgroup. Let $V$ be a rational representation of $\widetilde{G}$ and hence of $G$ as well. Let $v \in V$ and define $\widetilde{G}_v := \{g \in \widetilde{G}\ |\ gv = v\}$. Let $\mathcal{T}$ be a set of maximal tori of $G$ such that for every parabolic $P$ with $gPg^{-1} = P$ for all $g \in \widetilde{G}$, there exists $T \in \mathcal{T}$ such that $T \subseteq P$.
Then we have
\begin{enumerate}
\item $v$ is $G$-polystable $\Longleftrightarrow \overline{O_{T,v}} \subseteq O_{G,v}\  \mbox{ for all }T \in \mathcal{T}$;
\item $v$ is $G$-semistable $\Longleftrightarrow v$ is $T$-semistable for all $T \in \mathcal{T}$.
\end{enumerate}
\end{theorem}

\begin{proof}
Let us first prove part $(1)$. The $\implies$ implication is clear. We now prove the backwards implication by proving the contrapositive. Suppose $v$ is not $G$-polystable. Let $S = O_{G,w}$ be the unique closed orbit in $\overline{O_{G,v}}$. Then consider $P_{S,v}$, which satisfies $g P_{S,v} g^{-1} = P_{S,v}$ for all $g \in \widetilde{G}_v$ by Proposition~\ref{prop:improve-kempf}. Thus, there exists $T \in \mathcal{T}$ such that $T \subseteq P_{S,v}$. Hence, for some one-parameter subgroup $\lambda$ of $T$ we have  $\lim_{t \to 0} \lambda(t)v \notin O_{G,v}$. Thus, $\lim_{t \to 0} \lambda(t)v \in \overline{O_{T,v}} \setminus O_{G,v}$, so we get $\overline{O_{T,v}} \nsubseteq O_{G,v}$. 

The second part is analogous, where you take instead $S = \{0\}$ which is indeed the unique closed orbit in $\overline{O_{G,v}}$ if $v$ is not semistable.

\end{proof}

\section{Representations of (products) of special linear groups} \label{sec:rep-theory-background}
In this section, we collect a few results that will help with explicit computations when the acting group is a special linear group or a product of special linear groups. We first recall briefly the connection between parabolic subgroups of $\SL(V)$  and flags in $V$. Then, we discuss flags of $H$-stable subspaces of $V$ for a special class of subgroups $H \subseteq \SL(V)$ for which $V$ is a semisimple $H$-module. This will be very useful for computations. Finally, we give a quick proof of Theorem~\ref{thm:evenpoly} on the semistability/polystability of entirely even polynomials.

\subsection{Parabolic subgroups of $\SL(V)$}
The results in the above section warrant a brief discussion about parabolic subgroups and their maximal tori so that one can use them for computational purposes. We state results without proof referring the reader to standard texts \cite{Fulton, Fulton-Harris, Weyman} for details.

Let $V$ be an $n$-dimensional vector space. A flag $\mathcal{F}$ is a sequence of subspaces $0 = F_0 \subseteq F_1 \subseteq F_2 \subseteq \dots \subseteq F_k = V$. We do not restrict the dimensions of $F_i$, the inclusions force them to be an increasing sequence. Associated to a flag is a parabolic subgroup $P_\mathcal{F}$ of $\SL(V)$ defined by 
$$
P_\mathcal{F} = \{g \in \SL(V) \ |\ g F_i = F_i \ \forall \ i \in [k]\}.
$$

To each basis $\mathcal{B} = (b_1,\dots,b_n)$ of $V$, we define a maximal torus $T_\mathcal{B}$ of $V$ consisting of all $g \in \SL(V)$ such that each $b_i$ is an eigenvector when viewing $g$ as a linear transformation from $V$ to $V$. Clearly, permuting the basis does not change $T_\mathcal{B}$. Using the basis $\mathcal{B}$, one can identify $V$ with $K^n$ and consequently $\SL(V)$ with $\SL_n$. Under this identification, $T_\mathcal{B}$ is just the standard diagonal torus, i.e., the subgroup of all diagonal $n \times n$ matrices (with determinant $1$).

A basis $\mathcal{B} = (b_1,\dots,b_n)$ is called {\em compatible} with the parabolic $P_\mathcal{F}$ if each $F_i$ is a coordinate subspace in the basis $\mathcal{B}$, i.e., it is spanned by a subset of the basis. In this case, we will also say $\mathcal{B}$ is compatible with the flag $\mathcal{F}$. For a basis $\mathcal{B}$ that is compatible with a parabolic $P_\mathcal{F}$, the maximal torus $T_\mathcal{B} \subseteq P_{\mathcal{F}}$.  Further, all maximal tori of $P_\mathcal{F}$ arise from a compatible basis. 

If $G = \SL(V_1) \times \SL(V_2) \times \dots \times \SL(V_k)$, then any parabolic subgroup $P$ of $G$ is of the form $P = P^{(1)} \times P^{(2)} \times \dots \times P^{(k)}$ where each $P^{(i)}$ is a parabolic subgroup of $\SL(V_i)$. Thus, a collection of flags $\bm{\mathcal{F}} = (\mathcal{F}^{(1)},\dots,\mathcal{F}^{(k)})$ defines a parabolic subgroup for $G$. A maximal torus $T$ of $P_{\bm{\mathcal{F}}}$ is a product of maximal tori $T = T^{(1)} \times \dots \times T^{(k)}$ where each $T^{(i)}$ is a maximal torus for $V_i$. Thus a collection of compatible basis $\bm{\mathcal{B}} = (\mathcal{B}^{(1)},\dots,\mathcal{B}^{(k)})$, where each $\mathcal{B}^{(i)}$ is a basis of $V_i$, defines a maximal torus of $P_{\bm{\mathcal{F}}}$.

\subsection{Complete reducibility}
In order to use the results in the previous section, we often want to investigate parabolic subgroups containing the isotropy subgroup of a point. If the group acting is $\GL(V)$ or $\SL(V)$, this amounts to investigating flags of subspaces stable under the isotropy subgroup. Hence, in this section, we collect a few results on flags of $H$-stable subspaces in $V$ for some subgroup $H \subseteq \GL(V)$ in the special case where $V$ is a semisimple $H$-module, i.e., $V$ is completely reducible as a $H$-module. 

\begin{remark} \label{Rem:G-cr}
Then notion of $G$-complete reducibility was introduced by Serre \cite{Serre}. For a reductive group $G$, a (closed) subgroup $H$ is called $G$-completely reducible ($G$-cr for short) if for every parabolic subgroup $P$ of $G$ that contains $H$, there exists a Levi subgroup of $P$ that contains $H$. For $G = \GL(V)$, a subgroup $H$ is $G$-cr if and only if $V$ is a semisimple $H$-module. In particular, if $H$ is linearly reductive (for e.g., a torus or a finite group whose order is not a multiple of the characteristic), then it is automatically $G$-cr for $G = \GL(V)$. 
\end{remark}

For this section, let $G = \GL(V)$ and let $H$ be a $G$-cr subgroup of $G$, i.e., $V$ is a semisimple $H$-module. We have a decomposition into isotypic components
$$
V = E_1 \oplus E_2 \oplus \dots \oplus E_r.
$$

where each $E_i \cong V_i^{m_i}$ for some irreducible representation $V_i$ of $H$ (where $V_i$ and $V_j$ are non-isomorphic for $i \neq j$). Let $\dim(V_i) = n_i$.
For each $E_i$, fix an isomorphism that allows us to identify $E_i = V_i \otimes K^{m_i}$.

\begin{definition}
Let $W_1,\dots,W_r$ be vector spaces and let $U = \oplus_i W_i$. For a collections of flags $\mathcal{F}^{(i)}$ of $W_i$, we define their direct sum $\mathcal{F} = \oplus \mathcal{F}^{(i)}$, a flag of $U$ by setting $F_j = \oplus_i F^{(i)}_j$ for all $j$.
\end{definition}

\begin{lemma}
Let $\mathcal{F} = 0 \subseteq F_1 \subseteq F_2 \subseteq \dots \subseteq F_t = V$ be a flag of $H$-stable subspaces. For each $i$, the restricted flag $\mathcal{F}|_{E_i} :=  0 \subseteq F_1\cap E_i \subseteq F_2 \cap E_i \subseteq \dots \subseteq F_t \cap E_i = E_i$ is a flag of $H$-stable subspaces of $E_i$. Further, $\mathcal{F} = \oplus_i \mathcal{F}|_{E_i}$.
\end{lemma}

\begin{proof}
This follows from the fact that any subrepresentation $W$ of $V$ has the property that $W = \oplus_i (W \cap E_i)$ (which follows from complete reducibility). 
\end{proof}

The crucial point that comes from the above lemma is that to understand flags of $H$-stable subspaces of $V$, we can study each isotypic component separately and the following corollary is immediate.

\begin{corollary} \label{cor:flag-basis-combine-component}
Let $\mathcal{F} = 0 \subseteq F_1 \subseteq F_2 \subseteq \dots \subseteq F_t = V$ be a flag of $H$-stable subspaces. 
For each $i$, let $\mathcal{B}_i$ be a compatible basis for each $\mathcal{F}|_{E_i}$. Then $\cup_i \mathcal{B}_i$ is a compatible basis for $\mathcal{F}$. 
\end{corollary}

The following lemma investigates compatible bases for a single isotypic component. Let $\mathcal{I}_m \subseteq \GL_m(K)$ be a collection of invertible $m \times m$ matrices with the property that for any linearly independent collection $l_1,\dots,l_{m-1}$ of vectors in $K^m$, there exists $A \in \mathcal{I}_m$ such that every $l_i$ appears as some column of $A$ up to a non-zero scaling. For an $m \times m$ matrix $A$, we denote by $a_i \in K^m$ its $i^{th}$ column vector.

\begin{lemma} \label{lem:flag-basis-iso-component}
Let $W$ be an irreducible representation of $H$, and let $E = W \otimes K^m$. Let $\mathcal{F}$ be a flag of $H$-stable subspaces of $E$. Let $w_1,\dots,w_p$ be a basis of $W$. Then, for some $A \in \mathcal{I}_m$, the basis $\mathcal{B}(A) := \{w_i \otimes a_j\ |\ 1 \leq i \leq p, 1 \leq j \leq m\}$ is a compatible basis for $\mathcal{F}$.
\end{lemma}

\begin{proof}
Every $H$-stable subspaces of $E$ is of the form $W \otimes C$ for some subspace $C \subseteq K^m$. Thus, $\mathcal{F} = W \otimes \mathcal{F}'$ for some flag $\mathcal{F}'$ of $K^m$. Now, pick a compatible basis for $\mathcal{F}'$ step by step, i.e., pick a basis for $F'_1$, extend to a basis of $F'_2$ and so on until $F'_{t-1}$, where $t$ is such that $F'_{t-1} \subsetneq F'_t = K^m$. Note that $\dim(F'_{t-1}) = q \leq m-1$, and suppose $l_1,l_2,...,l_q$ is the basis of $F'_{t-1}$ that was picked so far. Now, there exists $A \in \mathcal{I}_m$ such that $l_i$ is a column of $A$ for all $i$ (up to a non-zero scaling). For such an $A$, it is clear that $\mathcal{B}(A)$ is compatible basis for $\mathcal{F}$.

\end{proof}

\begin{remark} \label{rem:I2}
One can take the set $\left\{\begin{pmatrix} 1 & 0 \\ a & 1 \end{pmatrix}\ |\ a \in K \right\}$ for $\mathcal{I}_2$. This is a simple but crucial observation that is needed in our computations for proving polystability of points required for exponential degree lower bounds. 
\end{remark}

\subsection{Polystability for polynomials}
In this section, we give a quick proof of Theorem~\ref{thm:evenpoly}.

\begin{proof} [Proof of Theorem~\ref{thm:evenpoly}]
Recall that $\kar(K) \neq 2$. Let $x_1,\dots,x_n$ denote the standard basis of $K^n$. Let $f \in W = \Sym^d(K^n)$ be an entirely even polynomial. We want to apply Theorem~\ref{thm:main3} for polystability and Theorem~\ref{thm:main.nullcone} for semistability. Let $G = \SL_n$ and $\widetilde{G} = \{A \in \Mat_{n,n}\ | \det(A) = \pm 1\} \subseteq \GL_n$ acting on $W$ in the natural way. Consider the action of the group $\{\pm 1\}^n$ (i.e., $(\Z/2)^n$) on $K^n$ given by $(t_1,\dots,t_n) \cdot (v_1,\dots,v_n) = (t_1v_1,\dots,t_nv_n)$. This action is given a map $(\Z/2)^n \rightarrow \widetilde{G}$. Let the image of this map be $H$. It is easy to see that $H \subseteq \widetilde{G}_f$.

We want to apply Theorem~\ref{thm:main3}. So, now we claim that $\mathcal{T} = \{\ST_n\}$ satisfies the hypothesis of Theorem~\ref{thm:main3}. Indeed, observe that if $P = P_{\mathcal{F}}$ is a parabolic such that $gPg^{-1} = p$ for all $g \in \widetilde{G}_f$, then we have $gPg^{-1} = P$ for all $g \in H$, which means that $\mathcal{F}$ is a flag of $H$-stable subspaces. Now, observe that $H$-stable subspaces are precisely coordinate subspaces.\footnote{This requires $\kar(K) \neq 2$. Note that if $\kar(K) = 2$, then $H$ is the trivial subgroup.} Thus, $\mathcal{F}$ is a flag of coordinate subspaces, which means that the standard basis is compatible with it, so $\ST_n \subseteq P = P_\mathcal{F}$.


Thus, we apply Theorem~\ref{thm:main3} to get that $f$ is $G$-polystable if and only if $f$ is $\ST_n$-polystable and that $f$ is $G$-semistable if and only if $f$ is $\ST_n$-semistable. The Theorem now follows from Corollary~\ref{cor:symm-torus}.
\end{proof}

\section{Closed orbits for degree lower bound purposes} \label{sec:closed-orbits-degree-bounds}
Before we go into the computational details, we need one quick observation.

\begin{remark} \label{Rem:torus-basechange-polystable}
Let $W$ be a rational representation of a reductive group $G$ and let $T$ be a maximal torus. Let $w \in W$. Then $w$ is $gTg^{-1}$ polystable/semistable/stable if and only if $g^{-1}w$ is $T$-polystable/semistable/stable.

Let $G = \SL(V)$ with a preferred basis $\mathcal{E} = \{e_1,\dots,e_n\}$. For any basis $\mathcal{B} = (b_1,\dots,b_n)$, we associate a maximal torus $T_\mathcal{B}$ consisting of all matrices which are diagonal with respect to this basis. Equivalently $T_\mathcal{B} = \{g \in SL(V) \ |\  \forall i, \ b_i \text{ is an eigenvector for } g\}$. We denote by $T$, the maximal torus $T_\mathcal{E}$. Let $L_\mathcal{B}$ be the linear transformation that sends $e_i \mapsto b_i$. Then $T_\mathcal{B} = L_\mathcal{B} T L_\mathcal{B}^{-1}$. Finally for some representation $W$ of $\SL(V)$, we have that $w$ is $T_\mathcal{B}$ polystable/semistable/stable if and only if $L_\mathcal{B}^{-1} w$ is $T_\mathcal{E}$ polystable/semistable/stable.
\end{remark}

\subsection{Closed orbit for cubic forms}

Let $\mathcal{E} = \{x_i,y_i,z_i\}_{1 \leq i \leq n}$ be the preferred basis for a $3n$-dimensional vector space $V$. Let $W = \Sym^3(V)^{\oplus 2}$ with the natural action of $G = \SL(V)$. Let $w = (\sum_i x_i^2 z_i, \sum_i y_i^2 z_i) \in W$.

\begin{proposition} \label{cubic-polystable}
The point $w \in W$ is $\SL(V)$-polystable.
\end{proposition} 

Consider the action on an $n$-dimensional torus $(K^*)^n$ on $V$ given by $(t_1,\dots,t_n) \cdot x_i = t_i x_i$,  $(t_1,\dots,t_n) \cdot y_i = t_i y_i$, $(t_1,\dots,t_n) \cdot z_i = t_i^{-2} z_i$. There is also an action of $S_n$ on $V$ that permutes the $x_i$, $y_i$ and $z_i$, i.e, $\sigma \cdot x_i = x_{\sigma(i)}, \sigma \cdot y_i = y_{\sigma(i)}$ and $\sigma \cdot z_i = z_{\sigma(i)}$. Combining the two actions, we get a map $\rho: (K^*)^n \rtimes S_n \rightarrow \SL(V) \subseteq \GL(V)$. Let $H := \rho((K^*)^n \rtimes S_n) \subseteq \GL(V).$

 Let $X = \spa\{x_i: i \in [n]\}, Y = \spa\{y_i: i \in [n]\}$ and $Z = \spa\{z_i: i \in [n]\}$. 

\begin{lemma} \label{lem-H-cr-cubic}
$V$ is a semisimple $H$-module.
\end{lemma}

\begin{proof}
All we need to do is to write $V$ as a direct sum of irreducible $H$-modules. Clearly $X \oplus Y \oplus Z = V$, so it suffices to show that each of $X,Y$ and $Z$ are irreducible $H$-modules. Let us do this for $X$. The others are similar. Suppose $0 \neq U \subsetneq X$ was a $H$-submodule. So, $U$ must be stable under the action of the torus $(K^*)^n$ (which is linearly reductive), which means that $U$ must be $\spa \{x_i : i \in I\}$ for some $\emptyset \neq I \subsetneq [n]$. But then, $U$ must also be stable under the action of $S_n$, which is not possible. Thus, no such $U$ exists and $X$ is irreducible.
\end{proof}

In the above proof, observe that $X$ and $Y$ are isomorphic $H$-modules, so we conclude the following:

\begin{corollary}
Let $P = {\rm span}\{x_i,y_i : i \in [n]\}$ and $Q = {\rm span}\{z_i: i \in [n]\}$. Then $V = P \oplus Q$ is the isotypic decomposition of $V$ with respect to $H$. 
\end{corollary}

We now turn to finding compatible basis for flags of $H$-stable subspaces. First, for $a \in K$, let us define $\mathcal{B}_a := \{x_i + a y_i, y_i, z_i : 1\leq i \leq n\}$.

\begin{lemma} \label{lem:cubic-compatible-basis}
Let $\mathcal{F}$ be a flag of $H$-stable subspaces of $V$. Then there exists $a \in K$ such that $\mathcal{B}_a$ is a compatible basis for $\mathcal{F}$.
\end{lemma}

\begin{proof}
By Corollary~\ref{cor:flag-basis-combine-component}, to find a compatible basis for $\mathcal{F}$, it suffices to find a compatible basis for $\mathcal{F}|_P$ and $\mathcal{F}|_Q$ separately. Since $Q$ itself is irreducible, the flag $\mathcal{F}|_Q$ must be trivial and any basis will do. We pick $\{z_1,\dots,z_n\}$. 

The isotypic component $P \cong X^{\oplus 2}$. So, by Lemma~\ref{lem:flag-basis-iso-component} and the choice of $\mathcal{I}_2$ in Remark~\ref{rem:I2}, we deduce that there is an $a \in K$ such that $\{x_i + a y_i, y_i : i \in [n]\}$ is a compatible basis. For this choice of $a$, we conclude that $\mathcal{B}_a = \{x_i + a y_i, y_i, z_i : i \in [n]\}$ is a compatible basis by Corollary~\ref{cor:flag-basis-combine-component}.
\end{proof}

\begin{lemma} \label{lem:w-T-E-polystable}
$w$ is $T_\mathcal{E}$-polystable.
\end{lemma}

\begin{proof}
Write $w = (w_1,w_2)$. Then $w_1 = \sum_i x_i^2z_i$ and $w_2 = \sum_i y_i^2 z_i$ are both weight decompositions. Further, it's an easy check to see that the sum of weights $\sum_i {\rm wt}(x_i^2z_i) + {\rm wt}(y_i^2 z_i) = 0$, which means that $0$ is in the relative interior of the weight polytope, and so $w$ is $T_\mathcal{E}$-polystable.
\end{proof}

\begin{lemma} \label{torus-polystable-cubic}
For all $a \in K$, $w$ is $T_{\mathcal{B}_a}$-polystable. 
\end{lemma}

\begin{proof}
Let $L_a$ be the linear transformation that takes $x_i \mapsto x_i + a y_i$ and keeps $y_i,z_i$ invariant for all $i$.
It suffices to prove that $L_a^{-1}(w) = L_{-a}(w)$ is $T_\mathcal{E}$-polystable by Remark~\ref{Rem:torus-basechange-polystable}.
Observe that $L_{-a}$ sends $x_i^2z_i \mapsto x_i^2 z_i - 2a x_iy_iz_i + a^2y_i^2z_i$ and $y_i^2 z_i \mapsto y_i^2 z_i$. Observe that ${\rm wt}(x_iy_iz_i) = \frac{1}{2} {\rm wt}(x_i^2 z_i) + \frac{1}{2}{\rm wt}(y_i^2z_i)$. This means that the weight polytope of $L_{-a}(w)$ is the same as the weight polytope of $w$ (even though the weights occuring in their weight decompositions may not be the same). Since weight polytopes determine polystability, see Lemma~\ref{lem:polystable-wp}, we conclude that $L_{-a}(w)$ is $T_\mathcal{E}$-polystable since $w$ is $T_\mathcal{E}$-polystable.
\end{proof}

Now, we combine all the results to prove Proposition~\ref{cubic-polystable}

\begin{proof}[Proof of Proposition~\ref{cubic-polystable}]
We want to use Theorem~\ref{thm:main3}. Take $G = \SL(V)$ and $\widetilde{G} = \GL(V)$. Then clearly we have $H \subseteq \widetilde{G}_v$ where $H$ is defined as in the beginning of this section. Now, suppose we have a parabolic $P_\mathcal{F}$ that is fixed by all elements of $\widetilde{G}_v$. In particular, it is fixed by all elements of $H$, so $\mathcal{F}$ must be a flag of $H$-stable subspaces. Hence, for some $a$, the basis $\mathcal{B}_a$ is compatible with $\mathcal{F}$  by Lemma~\ref{lem:cubic-compatible-basis}. In short this means that the collection $\mathcal{T} = \{T_{\mathcal{B}_a}\ |\ a \in K\}$ satisfies the hypothesis of Theorem~\ref{thm:main3}. Since $w$ is $T_{\mathcal{B}_a}$-polystable for all $a \in K$ by Lemma~\ref{torus-polystable-cubic}, we get that $\overline{O_{T,w}} = O_{T,w} \subseteq O_{G,w}$ for all $T \in \mathcal{T}$. Thus, by Theorem~\ref{thm:main3}, we conclude that $w$ is $G$-polystable.
\end{proof}

\subsection{Closed orbits for tensor actions} \label{closedorbittensorsubsection}
The idea is very much similar to the one on cubic forms, but the computations get a little bit cumbersome. Yet, spotting certain patterns will make the computation much easier. For this section, let $U,V,W$ be a $3n$-dimensional spaces with basis $\{u_i^{(k)} \ |\ 1\leq i \leq 3, 1 \leq k \leq n\}$, $\{v_i^{(k)} \ |\ 1\leq i \leq 3, 1 \leq k \leq n\}$, and $\{w_i^{(k)} \ |\ 1\leq i \leq 3, 1 \leq k \leq n\}$ respectively. Consider the action of $\SL(U) \times \SL(V) \times \SL(W)$ on $(U \otimes V \otimes W)^{\oplus 4}$. Let $F = (F_1,F_2,F_3,F_4) \in (U \otimes V \otimes W)^{\oplus 4}$, where 

\begin{align*}
F_1 &= \sum_{k=1}^n u_1^{(k)} v_2^{(k)} w_3^{(k)} + u_2^{(k)} v_3^{(k)} w_1^{(k)} + u_3^{(k)}  v_1^{(k)} w_2^{(k)} \\
F_2 &= \sum_{k=1}^n u_2^{(k)} v_1^{(k)} w_3^{(k)} + u_1^{(k)} v_3^{(k)} w_2^{(k)} + u_3^{(k)}  v_2^{(k)} w_1^{(k)} \\
F_3 &= \sum_{k=1}^n u_1^{(k)} v_1^{(k)} w_3^{(k)} + u_2^{(k)} v_3^{(k)} w_2^{(k)} + u_3^{(k)}  v_1^{(k)} w_1^{(k)} \\
F_4 &= \sum_{k=1}^n u_2^{(k)} v_2^{(k)} w_3^{(k)} + u_1^{(k)} v_3^{(k)} w_1^{(k)} + u_3^{(k)}  v_2^{(k)} w_2^{(k)}
\end{align*}

\begin{proposition}\label{tensor-polystable}
The point $F \in (U \otimes V \otimes W)^{\oplus 4}$ is $\SL(U) \times \SL(V) \times \SL(W)$-polystable.
\end{proposition}

Let us define a map $\phi_U:  ((\C^*)^3)^n \rightarrow \GL(U)$. To define such a map it suffices to understand the action of $t = (p_1,q_1,r_1,p_2,q_2,r_2,\dots,p_n,q_n,r_n)$ on each basis vector $b \in \B_u$. The map $\phi_U$ is defined by 
$$
\phi_U(t) u_1^{(k)} = p_k u_1^{(k)}, \phi_U(t) u_2^{(k)} = p_k u_2^{(k)} \text{ and } \phi_U(t) u_3^{(k)} = (q_kr_k)^{-1} u_3^{(k)}.
$$ 
Similarly define $\phi_V: ((\C^*)^3)^n \rightarrow \GL(V)$ by 
$$
\phi_V(t) v_1^{(k)} = q_k v_1^{(k)}, \phi_V(t) v_2^{(k)} = q_k v_2^{(k)} \text{ and } \phi_V(t) v_3^{(k)} = (p_kr_k)^{-1} v_3^{(k)}.
$$
 Finally, define $\phi_W: ((\C^*)^3)^n \rightarrow \GL(W)$ by 
 $$
 \phi_W(t) w_1^{(k)} = r_k w_1^{(k)}, \phi_W(t) w_2^{(k)} = r_k w_2^{(k)} \text{ and } \phi_W(t) w_3^{(k)} = (p_kq_k)^{-1} w_3^{(k)}.
 $$
 
 Let $\phi = (\phi_U,\phi_V,\phi_W): ((\C^*)^3)^n \rightarrow \GL(U) \times \GL(V) \times \GL(W)$. There is also an action of $S_n$ on $U, V$ and $W$ defined by $\sigma(u_i^{(k)}) = u_i^{(\sigma(k))}$, $\sigma(v_i^{(k)}) = v_i^{(\sigma(k))}$, and $\sigma(w_i^{(k)}) = w_i^{(\sigma(k))}$ respectively. That action gives a map $\psi: S_n \rightarrow \GL(U) \times \GL(V) \times \GL(W)$. Put together, we get a map $\phi \rtimes \psi : ((\C^*)^3)^n \rtimes S_n \rightarrow \GL(U) \times \GL(V) \times \GL(W)$. Let $H$ denote the image of $\phi \rtimes \psi$.

\begin{lemma}
$U,V,W$ are all semisimple $H$-modules.
\end{lemma} 

\begin{proof}
We will only prove this for $U$, the others are similar. For $i = 1,2,3$, let $X_i = \spa( u_i^{(k)} : k \in [n])$. Then, $X_i$ is an irreducible representation of $H$, which can be seen by an argument similar to the one in the proof of Lemma~\ref{lem-H-cr-cubic}. Clearly $X_1 \oplus X_2 \oplus X_3 = U$, so $U$ is semisimple.
\end{proof}

Moreover, observe that in the proof of the above lemma, $X_1 \cong X_2 \ncong X_3$. In particular, the isotypic decomposition of $U = P \oplus X_3$ where $P = X_1 \oplus X_2$. For $a \in K$, define the basis $\mathcal{B}_{U,a} = \{u_1^{(k)} + a u_2^{(k)}, u_2^{(k)}, u_3^{(k)} \ |\ k \in [n]\}$ of $U$. In studying flags of $H$-stable subspaces of $U$, the situation is similar to that of Lemma~\ref{lem:cubic-compatible-basis}. Hence, we can conclude that for any flag of $H$-stable subspaces of $U$, there is a compatible basis of the form $\mathcal{B}_{U,a}$ for some $a \in K$. Similar arguments hold for $V$ and $W$ where $\mathcal{B}_{V,b}$ and $\mathcal{B}_{W,c}$ for $b,c \in K$ are defined analogously. 
We define $\mathcal{B}_{a,b,c} = \mathcal{B}_{U,a} \otimes \mathcal{B}_{V,b} \otimes \mathcal{B}_{W,c} = \{z \otimes z' \otimes z''\ | z \in \mathcal{B}_{U,a}, z' \in \mathcal{B}_{V,b}, z'' \in \mathcal{B}_{W,c}\}$. 

From the above discussion, we conclude:
\begin{lemma}
Suppose $(\mathcal{F}_1,\mathcal{F}_2, \mathcal{F}_3)$ is a $3$-tuple of $H$-stable of flags of $U,V$ and $W$ respectively. Then, there exists a compatible basis of the form $\mathcal{B}_{a,b,c}$
\end{lemma}

Let $L_{U}^a$ be the linear transformation that sends $u_1^{(k)} \mapsto u_1^{(k)} + a u_2^{(k)}$ and leaves $u_2^{(k)}$ and $u_3^{(k)}$ invariant. Note that $(L_U^a)^{-1} = L_U^{-a}$. Similarly, define $L_V^b$ and $L_W^c$. 
For $a,b,c \in K$, define:
$$
L^{a,b,c} = L_{U}^a \otimes L_{V}^b \otimes L_{W}^c : U \otimes V \otimes W \rightarrow U \otimes V \otimes W.
$$

For $S = (S_1,\dots,S_r) \in (U \otimes V \otimes W)^{\oplus r}$, write each $S_i = \sum d(i)_{a,b,c}^{k_a,k_b,k_c}  u_a^{(k_a)} v_b^{(k_b)} w_c^{(k_c)}$. Define the support $\supp(S_i) := \{u_a^{(k_a)} v_b^{(k_b)} w_c^{(k_c)} \ | d(i)_{a,b,c}^{k_a,k_b,k_c} \neq 0\}$ and define the total support $\tsupp(S) = \cup_i \supp(S_i)$.

Now, consider $F = (F_1,F_2,F_3,F_4) \in (U \otimes V \otimes W)^{\oplus 4}$. 

\begin{lemma} \label{lem:tsupp}
Fix $a,b,c \in K$, let $L = L^{a,b,c}$ and let $F' = L(F)$. Then we have $\tsupp(F) = \tsupp(F')$.
\end{lemma}

\begin{proof}
One way to prove this lemma is by brute force computation, for example, with the use of a computer. However, we will give a proof by spotting key patterns. Let 
\begin{align*}
\mathcal{C}_{1,k} & = \{u_1^{(k)} v_2^{(k)} w_3^{(k)},u_2^{(k)} v_1^{(k)} w_3^{(k)}, u_1^{(k)} v_1^{(k)} w_3^{(k)}, u_2^{(k)} v_2^{(k)} w_3^{(k)}\} = \{u_i^{(k)} v_j^{(k)} w_3^{(k)} \ |\ i,j \in \{1,2\}\},\\
\mathcal{C}_{2,k} & = \{u_i^{(k)} v_3^{(k)} w_j^{(k)} \ |\ i,j \in \{1,2\}\}, \text{ and} \\
\mathcal{C}_{3,k} & = \{u_3^{(k)} v_i^{(k)} w_j^{(k)}\ |\ i,j \in \{1,2\}\}.
\end{align*}

Observe that each $F_i$ is a sum of monomials, exactly one from each $\mathcal{C}_{i,k}$. In particular, $\tsupp(F) = \cup_{i,k} \mathcal{C}_{i,k}$. Also, observe that $L$ keeps the span of each $\mathcal{C}_{i,k}$ invariant. Moreover, observe that for a monomial $m \in \mathcal{C}_{i,k}$, we have $L(m) = m + \sum_{n \in \mathcal{C}_{i,k} \setminus \{m\}} \lambda_n n$ for some scalars $\lambda_n \in K$. Now, suppose $m$ occurs in $F_j$, then as observed above, none of the monomials in $\mathcal{C}_{i,k} \setminus \{m\}$ occur in $F_j$. This means that $\supp(F_j) \subseteq \supp(L(F_j)) \subseteq \tsupp(F)$. Since this holds for arbitrary $j$, we have $\tsupp(L(F)) = \tsupp(F)$ as required.
\end{proof}

\begin{lemma}
The point $F$ is $T_{\mathcal{E}}$-polystable.
\end{lemma} 
 
\begin{proof}
The argument is similar to the one in the proof of Lemma~\ref{lem:w-T-E-polystable} since a convex combination of weights in the weight space decomposition is $0$, see e.g., the computation in the proof of \cite[Proposition~8.1]{DM-exp}. 
\end{proof} 
 
 \begin{lemma} \label{lem:F-ps-alltorus}
 The point $F$ is $T_{\mathcal{B}_{a,b,c}}$-polystable for all $a,b,c \in K$.
 \end{lemma}
 
 \begin{proof}
 To check that $F$ is $T_{\mathcal{B}_{a,b,c}}$-polystable, it suffices to check that $L(F)$ is $T_\mathcal{E}$-polystable, where $L = L^{-a,-b,-c}$. Lemma~\ref{lem:tsupp} shows that both $F$ and $L(F)$ have the same weight sets and hence the same weight polytopes, so $L(F)$ is $T_{\mathcal{E}}$-polystable since $F$ is by the above lemma.
 \end{proof}
 
 \begin{proof} [Proof of Proposition~\ref{tensor-polystable}]
 This is very similar to the proof of Proposition~\ref{cubic-polystable}. Using similar arguments, we see that the collection $\mathcal{T} = \{T_{\mathcal{B}_{a,b,c}} : a,b,c \in K\}$ satisfies the hypothesis of Theorem~\ref{thm:main3} and so by Lemma~\ref{lem:F-ps-alltorus}, we conclude that $F$ is $G$-polystable. We leave the details to the reader.
 \end{proof}

\section{Degree lower bounds via Grosshans principle} \label{sec:Grosshans}
In this section, we discuss our method to prove lower bounds, in particular we give a proof of Theorem~\ref{thm:liftingbounds}. Then, using Theorem~\ref{thm:liftingbounds} along with the results on polystability from the previous section, we give a proof of Theorems~\ref{thm:cubicbounds} and ~\ref{thm:tensorbounds}.

 The following lemma is crucial for our purposes, see 
\cite[Lemma~3.3]{BHM}.

\begin{lemma} \label{lem:BHM}
Let $V$ be a rational representation of a reductive group $G$ and let $v \in V$ and let $H = G_v := \{g \in G \ | \ gv = v \}$.
\begin{itemize}
\item The natural map $G/H \rightarrow G \cdot v$ is a homeomorphism, and an isomorphism of varieties if and only if the orbit map $G \rightarrow O_v$ is separable.
\item $O_v$ is affine if and only if $G/H$ is affine if and only if $H$ is reductive.
\end{itemize}
\end{lemma}

Moreover, observe that when $G/H$ is affine, it is clearly a categorical quotient and hence its coordinate ring is  equal to $K[G]^H$.

Let $V,W$ be rational representations of a reductive group $G$. Let $v \in V$ such that $O_v$ is closed. Let $H = G_v$. Since the orbit of $v$ is closed,  $H$ is a closed reductive subgroup.
Consider the following three morphisms of affine varieties. The first map is 
\begin{align*}
\iota: W &\hookrightarrow G/H \times W\\
w & \mapsto (eH,w)
\end{align*}

The second map is 
\begin{align*}
\pi: G/H \times W & \rightarrow O_v \times W \\
(gH,w) &\mapsto (gv,w)
\end{align*}

The last map is just the closed embedding
$$
j: O_v \times W \hookrightarrow V \times W
$$

Composing the three maps, we get $\phi := j \circ \pi \circ \iota : W \rightarrow V \times W$ given by $w \mapsto (v,w)$. For any morphism $\psi$ between affine varieties, we denote by $\psi^*$ the corresponding map on coordinate rings in the other direction.

\begin{lemma} \label{lem:degreenonincreasing}
The map $\phi^*$ is degree non-increasing, i.e., if $f \in K[V \times W]$, then $\deg(f) \geq \deg(\phi^*(f))$.
\end{lemma} 

\begin{proof}
This is straightforward.
\end{proof}

\begin{proposition}
The map $\phi^*$ restricts to a map on invariant rings $K[V \times W]^G \rightarrow K[W]^H$. 
\end{proposition}

\begin{proof}
For $h \in H$ and $w \in W$, we see that $\phi(hw) = (v,hw)$ and $\phi(w) = (v,w)$ are in the same $G$-orbit because $h \cdot (v,w) = (hv,hw) = (v,hw)$, which follows because $H = G_v$. Thus $\phi$ maps $H$-orbits into $G$-orbits. Hence, any $G$-invariant function pulls back under $\phi^*$ to a $H$-invariant function.

\end{proof}

\begin{proposition} \label{prop:septosep}
Let $\{f_i : i \in I\}$ be a separating subset of invariants for the action of $G$ on $V \times W$. Then $\{\phi^*(f_i) : i \in I\}$ is a separating subset of invariants for the action of $H$ on $W$.
\end{proposition}

\begin{proof}
Observe that $\phi^* = \iota^* \circ \pi^* \circ j^*$. First, $j$ is just the restriction to a closed $G$-stable subset and $\pi$ is a homeomorphism, so $\{\pi^* \circ j^* (f_i) : i \in I\}$ is a separating subset for the action of $G$ on $G/H \times W$. But now, Grosshans principle gives us an isomorphism $\iota^*: K[G/H \times W]^G \xrightarrow{\sim} K[W]^H$, which means that the categorical quotients $(G/H \times W) \!\!\gitsymbol\! G \cong W \!\!\gitsymbol\! H$, so a separating subset for the action of $G$ on $G/H \times W$ gives a separating subset for the action of $H$ on $W$ via the map $\iota^*$.
\end{proof}

\begin{proof} [Proof of Theorem~\ref{thm:liftingbounds}]
From Lemma~\ref{lem:degreenonincreasing} and Proposition~\ref{prop:septosep}, we get $\beta_{\rm sep}(G, V \times W) \geq \beta_{\rm sep}(H,W)$. The other two inequalities are straightforward, see Equation~\ref{eq:deg.bds.ineq}.
\end{proof}
	
\subsection{Null cone bounds for non-connected reductive groups}
In order to prove degree lower bounds for one action, our strategy is essentially to reduce it to bounds for invariants defining the null cone of a related action. However, for this strategy to work, we need to start somewhere, i.e., be able to prove exponential lower bounds for invariants defining the null cone for some action. As mentioned in the introduction, it is relatively easier to prove null cone bounds for torus actions. Hence, we want to look for a point with a closed orbit whose stabilizer is a torus. On the other hand, having a significant finite group in the stabilizer can greatly simplify and ease the computations needed to prove that the point in question has a closed orbit. Thus, we find points with closed orbits whose stabilizers are not a torus, but the extension of a torus by a finite group. However, this brings a new problem, i.e., we now need to understand null cone bounds for groups that are a little more general than tori. In this subsection, we will show that the finite group part does not affect the null cone bound adversely.

\begin{proposition} \label{prop:fingp.nullcone}
Let $V$ be a rational representation of a reductive group $G$. Let $G^\circ$ denote the identity component of $G$. Then $\sigma(G, V) \geq \sigma(G^\circ,V)$.
\end{proposition}

\begin{proof}
The Hilbert--Mumford criterion,  (see e.g., \cite{Mumford} or \cite[Theorem~2.5.3]{DK-book}) can be formulated in the following way -- the null cone for the action of a reductive group is the union of null cones for all of its maximal tori. Since maximal tori for $G$ are precisely the maximal tori for $G^\circ$, we conclude that the null cone for the action of $G$ and $G^\circ$ on $V$ are the same. Let $\mathcal{N} := \mathcal{N}(G,V) = \mathcal{N}(G^\circ, V)$. By definition of $\sigma(G,V)$, there exist $f_1,\dots,f_r \in K[V]^G$ such that $\mathbb{V}(f_1,\dots,f_r) = \mathcal{N}$ with $\max \{\deg(f_i)\} = \sigma(G,V)$. This means that $f_1,\dots,f_r$ are $G^\circ$-invariant functions that cut out the null cone $\mathcal{N}(G^\circ,V)$. Thus, $\sigma(G^\circ,V) \leq \max \{\deg(f_i)\} = \sigma(G,V)$. 
\end{proof}

\subsection{Cubic forms}
Assume $\kar(K) \neq 2$ for this subsection. Let $V$ be a $3n$-dimensional vector space with preferred basis $\mathcal{E} = \{x_i,y_i,z_i\}_{1 \leq i \leq n}$. Consider the action of $G = \SL(V)$ on $\Sym^3(V)$, and the diagonal action of $\SL(V)$ on $W = \Sym^3(V)^{\oplus 2}$.
Consider $w = (\sum_i x_i^2 z_i, \sum_i y_i^2z_i) \in W^{\oplus 2}$. Then, by Proposition~\ref{cubic-polystable}, we know that $O_w$ is closed. So, in order to use Theorem~\ref{thm:liftingbounds}, we need to compute $G_w$. Consider the action of the $n$-dimensional torus on $V$ as follows. For $t = (t_1,\dots,t_n) \in (\C^*)^n$, we have $t x_i = t_i x_i$, $ty_i = t_i y_i$, and $tz_i = t_i^{-2}z_i$. This action gives a map $\psi: (\C^*)^n \rightarrow \SL(V)$. Let $L := \psi((\C^*)^n)$.

\begin{lemma}
Let $g \in G_w$. Then $gx_i = c_i x_\sigma(i), gy_i = \pm c_i y_{\sigma(i)}$ and $gz_i = c_i^{-2} z_{\sigma(i)}$ for some scalars $c_i \in K$ and $\sigma \in S_n$.
\end{lemma}

\begin{proof}
The proof is the same as \cite[Corollary~7.11]{DM-exp}, the proof of which uses \cite[Lemma~7.8]{DM-exp} (a result which holds precisely when characteristic $\neq 2$ as can be easily seen from the proof). 
\end{proof}

\begin{corollary}
We have $G_w^\circ = L$.
\end{corollary}

\begin{proof}
The above lemma associates a permutation $\sigma$ to each $g \in G_w$. That gives a map which can easily be seen to be a group homomorphism, which we call $\phi: G_w \rightarrow S_n$. The kernel of $\phi$ is precisely all the elements of $G_w$ that keep $x_i, y_i, z_i$ invariant up to scalars. Moreover, by the previous lemma, for $g \in {\rm ker}(\phi)$, the action is given by $gx_i = c_i x_i, gy_i = \pm c_i y_i$ and $gz_i = c_i^{-2} z_i$ for some scalars $c_i \in K$. Now, it is easy to see that ${\rm ker}(\phi)$ contains $L$ and the quotient is a finite group (indeed just a subgroup of $(\Z/2)^n$). Since $L$ is connected, we conclude that ${\rm ker}(\phi)^{\circ} = L$. Clearly, since $G_w$ is a finite extension of ${\rm ker}(\phi)$, we deduce that $G_w^\circ = L$. 
\end{proof}

\begin{proof}[Proof of Theorem~\ref{thm:cubicbounds}]
Since $w = (\sum_i x_i^2 z_i, \sum_i y_i^2 z_i)$ has a closed $\SL(V)$-orbit, by Theorem~\ref{thm:liftingbounds}, we get that
$$
\beta(\SL(V), \Sym^3(V)^{\oplus 3}) \geq \sigma(G_w,\Sym^3(V)).
$$

But now, since $G_w^\circ = L$, we get that $\sigma(G_w, \Sym^3(V)) \geq \sigma(L,\Sym^3(V))$ by Proposition~\ref{prop:fingp.nullcone}. Since $L$ is a torus, the degree bound on generators as well as the bound on the degree of invariants defining the null cone only depend on the weight set (and in particular is oblivious to the characteristic). The computation one needs to do to obtain a lower bound on $\sigma(L,\Sym^3(V))$ is already done, see \cite[Corollary~7.4]{DM-exp} (where $L$ for us is denoted $H$). Thus, we conclude

$$
\beta(\SL(V), \Sym^3(V)^{\oplus 3}) \geq \sigma(G_w,\Sym^3(V)) \geq \sigma(L,\Sym^3(V)) \geq \frac{2}{3} (4^n-1).
$$
\end{proof}

\subsection{Tensor actions}
Let $U,V,W$ be $3n$-dimensional spaces with a preferred basis $\{u_i^{(k)}\}, \{v_i^{(k)}\}, $ and $\{w_i^{(k)}\}$ respectively as in Section~\ref{closedorbittensorsubsection}. Consider the action of $G = \SL(U) \times \SL(V) \times \SL(W)$ on $(U \otimes V \otimes W)^{\oplus 4}$. Let $F = (F_1,F_2,F_3,F_4) \in (U \otimes V \otimes W)^{\oplus 4}$ as in Section~\ref{closedorbittensorsubsection}, which has a closed $G$-orbit. 

We consider a slightly different group 
$$
J := \{(g_1,g_2,g_3) \in \GL(U) \times \GL(V) \times \GL(W)\ | \det(g_1) \det(g_2) \det(g_3) = 1\}.
$$

Both $G$ and $J$ are subgroups of $\GL(U) \times \GL(V) \times \GL(W)$ and act naturally on $(U \otimes V \otimes W)^{\oplus r}$ for $r \in \Z_{>0}$. As shown in \cite[Section~8]{DM-exp}, the orbits with respect to both groups are precisely the same and hence so are the invariant rings. Moreover, $J$ is a reductive group by Matsushima's criterion.


Now, we turn to computing the stabilizer $J_F$, or rather its identity component. Recall the map $\phi$ defined in Section~\ref{closedorbittensorsubsection}. Let $L := \phi( ((\C^*)^3)^n)$. 
\begin{lemma}
The subgroup $H = J_F^\circ$, the identity component of the stabilizer of $F$.
\end{lemma}

\begin{proof}
By Kruskal's uniqueness theorem \cite{Kruskal} (see also \cite{Landsberg-Kruskal}), any $g \in J_F$ permutes the terms in each of the $F_i$'s. By a similar argument to the one in the case of cubic forms, the subgroup of $J_F$ that fixes all monomials is of finite index in $J_F$. But this is precisely $L$ by the same arguments in \cite[Lemma~8.11]{DM-exp}. Since $L$ is connected, we must have $J_F^\circ = L$.
\end{proof}

\begin{proof} [Proof of Theorem~\ref{thm:tensorbounds}]
This is similar to the proof of Theorem~\ref{thm:cubicbounds}. We have
\begin{align*}
\beta(G,(U \otimes V \otimes W)^{\oplus 5}) &= \beta(J, (U \otimes V \otimes W)^{\oplus 5}) \\
& \geq \sigma(J_F, U \otimes V \otimes W)\\
& \geq \sigma(L, U \otimes V \otimes W) \\
& \geq 4^n-1
\end{align*}

The first equality follows from the fact that $G$-orbits and $J$-orbits are the same, so the corresponding invariant rings are also the same. The first inequality follows from applying Theorem~\ref{thm:liftingbounds} to the fact that $F$ has a closed orbit (by Proposition~\ref{tensor-polystable}). The second inequality follows from Proposition~\ref{prop:fingp.nullcone}. The last follows from the computation in \cite[Corollary~8.5]{DM-exp} (where $L$ for us is denoted $H$).
\end{proof}

\section{Polystability for symmetric polynomials} \label{sec:symm-polys}


In this section, we discuss stability notions for symmetric polynomials, in particular we give an algorithm to determine whether a symmetric polynomial is unstable, semistable, polystable, or stable. The techniques in this section go beyond the results stated in Section~\ref{sec:kempf}. Roughly speaking, in Section~\ref{sec:kempf}, the high-level idea was to check polystability (or similar) for a collection of maximal tori that covers all possible optimal parabolic subgroups. In this section, we will take a closer look at the parabolic itself and leverage that for a parabolic subgroup to be optimal, the associated optimal one-parameter subgroup must take a very specific form. So, we first discuss some generalities on one-parameter subgroups and their associated parabolics and then proceed to study the case of symmetric polynomials.

Let $\lambda: K^* \rightarrow \SL(V)$ be a $1$-parameter subgroup. Such a $1$-parameter subgroup is diagonalizable, i.e., we have a basis $v_1,\dots,v_n$ of $V$ (say $V$ is $n$-dimensional) such that $\lambda(t) v_i = t^{\beta_i} v_i$ for some $\beta_i \in \Z$. Without loss of generality, we can take $\beta_1 \geq \beta_2 \geq \dots \geq \beta_n$. Of course, some of the inequalities can be equalities. So, we must have $1 = k_0 < k_1 < k_2 < \dots < k_r = n+1$ such that $\beta_i = \beta_j$ for all $i,j \in \{k_{a-1}, k_{a-1} + 1,\dots, k_{a} - 1\}$ for any $a \in \{1,\dots,r\}$. Let $F_a$ denote the linear span of $v_1,v_2,\dots,v_{k_{a}-1}$. Let $\mathcal{F} = 0 \subseteq F_1 \subset F_2 \subset \dots \subset F_r = V$. Then, the parabolic associated to $\lambda$ is $P(\lambda) = P_\mathcal{F}$. An illustrative example is the following:

\begin{example} \label{eg:lambda-to-flag}
Let $x_1,x_2,x_3$ denote the standard basis of $K^3$. Consider the one-parameter subgroup $\lambda$ of $\SL_3$ given by $\lambda(t) = \begin{pmatrix} t^3 & 0 & 0 \\ 0 & t^3 & 0 \\ 0 & 0 & t^{-6} \end{pmatrix}$. Consider the flag $\mathcal{F} = 0 \subseteq \spa\{x_1,x_2\} \subseteq K^3$. Then $P(\lambda) = P_\mathcal{F} = \left(\begin{array}{cc|c} \ast & \ast & \ast \\ \ast & \ast & \ast \\ \hline 0 & 0 & \ast \end{array}\right) \subseteq \SL_3$.

\end{example}

\begin{definition}
Let $\mathcal{F} = (0 = F_0)  \subseteq F_1 \subseteq \dots \subseteq (F_r = V)$ be a flag. We call a tuple of subspaces $\underline{G} = (G_1,\dots,G_r)$ a splitting of $\mathcal{F}$ if $F_{i-1} \oplus G_i = F_i$ for all $i$. Observe that $\oplus_i G_i = V$. Denote by $\mathcal{S}_\mathcal{F}$ the set of splittings of $\mathcal{F}$.

Further, let $\underline{c} \in \Z^r$ be such that $\sum_i c_i \dim(G_i) = 0$. Then, we call $(\underline{G}, \underline{c})$ a decorated splitting of $\mathcal{F}$. We call $\underline{c}$ a decoration for the splitting $\underline{G}$.

Finally, for a decorated splitting $(\underline{G}, \underline{c})$, we define an associated $1$-parameter subgroup $\lambda = \lambda_{(\underline{G},\underline{c})}$ by $\lambda(t) v = t^{c_i} v$ for $v \in G_i$. 
\end{definition}

\begin{lemma} \label{lem:lambda-to-decoration}
Let $\mathcal{F}$ be a flag in $V$. Suppose $\lambda$ is a $1$-parameter subgroup such that $P(\lambda) = P_\mathcal{F}$. Then, there is a decorated splitting $(\underline{G},\underline{c})$ of $\mathcal{F}$ such that $\lambda = \lambda_{(\underline{G},\underline{c})}$.
\end{lemma}

\begin{proof}
Take a basis $v_1,\dots,v_n$ of $V$ such that $\lambda(t) v_i = t^{\beta_i} v_i$ for some $\beta_i \in \Z$ and assume w.l.o.g. that $\beta_1 \geq \beta_2 \geq \dots \geq \beta_n$. Let $1 = k_0 < k_1 < k_2 < \dots < k_r = n+1$ such that $\beta_i = \beta_j$ for all $i,j \in \{k_{a-1}, k_{a-1} + 1,\dots, k_{a} - 1\}$ for any $a \in \{1,\dots,r\}$. Then $P(\lambda) = P_\mathcal{F}$ means that $F_a$ is the linear span of $v_1,v_2,\dots,v_{k_a-1}$ as explained above (just before Example~\ref{eg:lambda-to-flag}.

So, now let $G_a$ to be the linear span of $v_{k_{a-1}}, v_{k_{a-1} + 1}, \dots, v_{k_a - 1}$ and let $c_a = \beta_{k_{a-1}} = \beta_{k_{a-1}+1} = \dots = \beta_{k_a-1}$. It is now straightforward to check that $\lambda = \lambda_{(\underline{G},\underline{c})}$.
\end{proof}

Perhaps the most important result for this section is the following lemma.

\begin{lemma} \label{lem:technical-symm}
Suppose $W$ is a representation of $G = \SL(V)$ and $w \in W$ such that $O_w$ is not closed. Let $S$ be a closed $G$-stable subset such that $S \cap O_w = \emptyset$ and $S \cap \overline{O_w} \neq \emptyset$. Let $\mathcal{F}$ be a flag of $V$ such that the optimal parabolic subgroup $P_{S,w}  = P_\mathcal{F}$. Then, there exists unique indivisible $\underline{c} = (c_1,\dots,c_r) \in \Z^r$ with $\sum_i c_i \dim(F_i/F_{i-1}) = 0$ and $c_1 \geq c_2 \geq \dots \geq c_r$ such that the map
\begin{align*}
\mathcal{S}_\mathcal{F} & \longrightarrow \Lambda(S,v) \\
\underline{G} &\longmapsto \lambda_{(\underline{G}, \underline{c})}
\end{align*}
is a bijection between splittings and optimal $1$-parameter subgroups.
\end{lemma}

\begin{proof}
Let $\underline{G}$ be a splitting. Then, take a basis $\mathcal{B}$ such that each $G_i$ is a coordinate subspace (i.e., span of a subset in $\mathcal{B}$). Then, by Theorem~\ref{thm:Kempf}, part $(3)$, there is an optimal $1$-parameter subgroup contained in $T_\mathcal{B}$. Let us call that $\lambda$. Fix $1 \leq i \leq r$. Let $A : G_i \rightarrow G_i$ be a linear transformation with determinant $1$. Let $L(A) : W\rightarrow W$ be the linear transformation that is identity on $G_j$ for $j \neq i$ and agrees with $A_i$ on $G_i$. It is easy to see that $L(A) \lambda L(A)^{-1}$ is also an optimal $1$-parameter subgroup in $T_\mathcal{B}$. Thus, we must have $L(A) \lambda L(A)^{-1} = \lambda$ for all $A \in \SL(G_i)$. It is straightforward to argue that this means there is $c_i \in \Z$ such that $\lambda(t) v = t^{c_i} v$ for all $v \in G_i$. In particular, this means that $\lambda$ does not depend on the choice of $\mathcal{B}$ or $T_\mathcal{B}$ but on just $G$ itself. Thus, to each splitting $\underline{G}$, we can associate a unique $\lambda = \lambda_{\underline{G},\underline{c}} \in \Lambda(S,w)$ (where apriori $\underline{c}$ depends on $\underline{G}$). Note that $\underline{c}$ is indivisible simply because $\lambda$ is optimal. 

To show that $\underline{c}$ does not depend on the choice of $\underline{G}$, we note that for any other splitting $\underline{G'}$, we have $p \in P_{S,w}$ such that $p \underline{G} = \underline{G'}$. This means that $p \lambda_{\underline{G}, \underline{c}} p^{-1} = \lambda_{\underline{G'},\underline{c}} \in \Lambda(S,w)$. This means that the choice of $\underline{c}$ is independent of the choice of $\underline{G}$. 

To summarize, we have shown that the existence of the map $\mathcal{S}_\mathcal{F} \rightarrow \Lambda(S,v)$. Injectivity is clear because you can recover $G_i$ from $\lambda$ uniquely as the subspace of $V$ on which $\lambda(t)$ acts by $t^{c_i}$. To show surjectivity is to show that any optimal $1$-parameter subgroup $\lambda \in \Lambda(S,w)$ arises as $\lambda_{\underline{G},\underline{c}}$ for some splitting $\underline{G}$ and $\underline{c}$. But this follows from Lemma~\ref{lem:lambda-to-decoration}.

\end{proof}

\begin{remark} \label{rem:2stepflag}
Suppose $\mathcal{F}$ is a 2-step flag, i.e., $\mathcal{F} = 0 \subseteq F_1 \subseteq F_2 = V$ and let $\underline{c}$ be as in Lemma~\ref{lem:technical-symm} above. Then $\underline{c}$ must be the indivisible integral vector that is a multiple of $(\dim(V) - \dim(F_1), -\dim(F_1))$.
\end{remark}

\begin{lemma}
Assume $n \geq 2$. Let $x_1,\dots,x_n$ denote the standard basis for $K^n$ and consider the natural action of $S_n$ on $K^n$ by permutation of $x_1,\dots,x_n$. Then $L = {\rm span}(x_1 + x_2 + \dots + x_n)$ and $M = \{\sum_i \alpha_i x_i \ |\ \sum_i \alpha_i = 0\}$ are the only non-trivial $S_n$-stable subspaces of $K^n$.
\end{lemma}

\begin{proof}
We will prove it by contradiction. Suppose $W$ is a non-trivial $S_n$-stable subspace that is neither $L$ nor $M$. Take $\alpha = (\alpha_1,\dots,\alpha_n) \in W$ such that it is not a multiple of $(1,1,\dots,1)$ (such an $\alpha$ exists because $W \neq L$). Without loss of generality, we can assume $\alpha_1 \neq \alpha_2$, so $\alpha - (12) \alpha = (d,-d,0,\dots,0)$ where $d = \alpha_1 - \alpha_2 \neq 0$. This means that $(1,-1,0,\dots,0) \in W$. It is easy to see then that $M \subseteq W$. Since $M$ has codimension $1$ and $W \neq M$, we must have $W = K^n$, which is a contradiction since we assumed $W$ is non-trivial.  
\end{proof}

\begin{corollary} \label{cor:sym.flags}
Assume $n \geq 2$. Let $\kar(K) = p$. Let $x_1,\dots,x_n$ denote the standard basis for $K^n$ and consider the natural action of $S_n$ on $K^n$ by permutation of $x_1,\dots,x_n$. Let $\mathcal{F}$ be a flag of $S_n$ stable subspaces. If $p \nmid n$, then $\mathcal{F}$ must be one of:
\begin{itemize}
\item $0 \subset K^n$;
\item $0 \subset L \subset K^n$;
\item $0 \subset M \subset K^n$.
\end{itemize} 

If $p \mid n$, then $\mathcal{F}$ must be one of:
\begin{itemize}
\item $0 \subset K^n$;
\item $0 \subset L \subset K^n$;
\item $0 \subset M \subset K^n$;
\item $0 \subset L \subset M \subset K^n$.
\end{itemize} 
Note that when $\kar(K) = p = 2$ and $n = 2$, then $L = M$. So, in this case, we only have two possible flags instead of four.
\end{corollary}

It is quite crucial to realize that Corollary~\ref{cor:sym.flags} is key to giving an algorithm for detecting polystability. Indeed, this shows that one has very few choices for an optimal parabolic subgroup, which narrows the search for an optimal one-parameter subgroup (if it exists). The rest of this section is devoted to discussing the algorithm to detect polystability of symmetric polynomials.

\subsection{The case $p \nmid n$}
Assume $n \geq 2$ for this subsection. Let $L,M \subseteq K^n$ be the two non-trivial $S_n$-stable subspaces as defined above. Then, it is easy to see that since $p \nmid n$, we have $L \oplus M = K^n$. In particular, this means that for the flag $0 \subset L \subset K^n$, a splitting is $(L,M)$ and for the flag $0 \subset M \subset K^n$, a splitting is $(M,L)$. Since both are $2$-step flags, the decoration is uniquely determined, it is $(n-1,-1)$ in the first instance and $(1,-(n-1))$ in the second instance. Let $\lambda_{\rm can}$ be the $1$-parameter subgroup of $\SL_n$ defined by 
\begin{equation}
\lambda_{\rm can}(t) \cdot v = \begin{cases} t^{n-1} v & \text{ if } v \in L \\ t^{-1} v & \text{ if } v \in M \end{cases}
\end{equation}
We call $\lambda_{\rm can}$ the canonical $1$-parameter subgroup for symmetric polynomials (in the case $p \nmid n$).

\begin{lemma} \label{pnmidn-1-onepsg}
Let $\kar(K) = p \nmid n$. Let $f \in K[x_1,\dots,x_n]^{S_n}_d$ be a degree $d$ symmetric polynomial. Then, $f$ is not polystable if and only if one of the two conditions hold:
\begin{itemize}
\item $\lim_{t \to 0} \lambda_{\rm can}(t) \cdot f$ exists and is not in $O_f$;
\item $\lim_{t \to \infty} \lambda_{\rm can}(t) \cdot f$ exists and is not in $O_f$;
\end{itemize}

Further, $f$ is unstable if and only if $\lim_{t \to 0} \lambda_{\rm can}(t) \cdot f = 0$ or $\lim_{t \to \infty} \lambda_{\rm can}(t) \cdot f = 0$.
\end{lemma}

\begin{proof}
Clearly if $f$ is polystable, then the two limits either do not exist or must be in $O_f$.

Now, suppose $f$ is not polystable, let $S = \overline{O_f} \setminus O_f$. Consider the optimal parabolic subgroup $P_{S,v}$. First, we claim that $P_{S,v}$ is not all of $\SL_n$. This is because then $P_{S,v} = P_\mathcal{F}$ where $\mathcal{F} = 0 \subseteq K^n$. Hence, the only possible splitting is $\underline{G} = (G_1) = (V)$ (i.e., $\mathcal{S}_\mathcal{F}$ is a singleton set). Now, consider an optimal one-parameter subgroup $\lambda = \lambda_{(\underline{G}, \underline{c})}$ as in Lemma~\ref{lem:technical-symm}. We must have $\underline{c} = (c_1) = 0$ because $ 0 = \sum_i c_i \dim G_i = c_1 \cdot n$, so $\lambda(t)$ is the trivial one-parameter subgroup, i.e., $\lambda(t)$ is the identity matrix for all $t$. Thus $\lim_{t \to 0} \lambda(t) \cdot f = f$, which contradicts the assumption that $f$ is not polystable and $\lambda$ is optimal.  

Thus $P_{S,v} = P_\mathcal{F}$ where $\mathcal{F}$ is either $0 \subset L \subset K^n$ or $0 \subset M \subset K^n$ by Corollary~\ref{cor:sym.flags}. In the former case $(L,M)$ is a splitting and by Remark~\ref{rem:2stepflag}, we see that $\lambda_{\rm can}$ is an optimal one-parameter subgroup. In the latter case, $(M,L)$ is a splitting and by Remark~\ref{rem:2stepflag}, $\lambda_{\rm can}^{-1}$ is an optimal one-parameter subgroup. Thus, one of these two one-parameter subgroups must drive $f$ out of its orbit in the limit, as required.

The argument for unstable is analogous, where you replace $S = \overline{O_f} \setminus O_f$ with $S = \{0\}$.
\end{proof}

Suppose that $p\nmid n$. For $f \in K[x_1,\dots,x_n]_d^{S_n}$, we can rewrite $f$ as a polynomial in $l = \sum_i x_i$ and $b_1,\dots,b_{n-1}$ where $b_i = x_i - x_{i+1}$. In other words, we have
\begin{equation} \label{eq:can-rewrite}
f = \sum_{i=0}^d l^i p_i,
\end{equation}
where $p_i$ is a polynomial in $b_1,b_2,\dots,b_{n-1}$.

\begin{remark}
We point out to the reader that we use $p$ for characteristic and $p_i$ to denote polynomials obtained by decomposing $f$ in a specific way as indicated above. These polynomials always come with a subscript which indicates their degree, so there is no scope for confusion. 
\end{remark}

\begin{theorem} \label{theo:pnmidn}
Let $\kar(K) = p \nmid n$. Let $f \in K[x_1,\dots,x_n]_d^{S_n}$. Write $f = \sum_i l^i p_i$ as described above. Then,  $f$ is unstable if and only if either $l^{\lfloor d/n \rfloor + 1} \mid f$ or $f = \sum_{i = 0}^{\lceil d/n \rceil - 1} l^i p_i$. Further, if $f$ is not unstable, then it is polystable unless the following conditions hold:
\begin{itemize}
\item $n \mid d$;
\item $l^{d/n} f_{d/n} \notin O_f$;
\item $l^{d/n} \mid f$ or $f = \sum_{i = 0}^{d/n} l^i p_i$
\end{itemize}
\end{theorem}

\begin{proof} 
Write $f = \sum_i l^i p_i$ as in Equation~\ref{eq:can-rewrite}.  $f$ is unstable precisely when $\lim_{t \to 0} \lambda_{\rm can}(t) f = 0$ or $\lim_{t \to \infty} \lambda_{\rm can}(t) f = 0$

 Observe that $\lambda_{\rm can}(t) f = \sum_i t^{(n-1)i - (d-i)} l^i p_i = \sum_i t^{ni-d} l^i p_i$. This limit as $t \to 0$ is $0$ if and only if $ni-d > 0$ for all $i$ such that $p_i \neq 0$. In other words, $p_i \neq 0 \implies i > d/n$, i.e., $i \geq {\lfloor d/n \rfloor + 1}$ since $i$ must be an integer. Thus $l^{\lfloor d/n \rfloor + 1} \mid f$. Similarly, the limit as $t \to \infty$ is $0$ if and only if $f = \sum_{i = 0}^{\lceil d/n \rceil - 1} l^i p_i$. 

If $f$ is not unstable, then it is semistable. Suppose $f$ is not polystable, then one of $\lim_{t \to 0} \lambda_{\rm can} (t) f$ and $\lim_{t \to \infty} \lambda_{\rm can}(t) f$ exists and is not in $O_f$.

Let's first suppose $\lim_{t \to 0} \lambda_{\rm can}(t) f$ exists and is not in $O_f$. If $n \nmid d$, for any $i$, we must have $ni - d < 0$ or $ni-d>0$. For $\lim_{t \to 0} \lambda_{\rm can}(t) f$ to exist, we must have $p_i = 0$ whenever $ni-d < 0$. Further, if $ni-d > 0$, then $t^{ni-d} l^i p_i$ will go to $0$ in the limit, i.e., $0 \in \overline{O_f}$, so $f$ is unstable, which is a contradiction. Hence, we must have $n \mid d$ and that $l^{d/n} \mid f$ (since $p_i = 0$ whenever $ni-d < 0$). Further, in this case, the limit is precisely $l^{d/n} f_{d/n}$.

The case $\lim_{t \to \infty} \lambda_{\rm can}(t) f$ exists and is not in $O_f$ is similar except we replace $l^{d/n} \mid f$ by $f = \sum_{i = 0}^{d/n} l^i p_i$. 
\end{proof}

The above results translate into the following algorithm:

\begin{algorithm} \label{algo:pnmidn}
 Now we give an algorithm that decides whether a symmetric polynomial is unstable/semistable/polystable/stable in the case $p \nmid n$.\\
\noindent{\bf Input:}  $f \in K[x_1,\dots,x_n]_d$
\begin{description}
\setlength\itemsep{.5em}
\item [Step 1] Write $f = \sum_i l^i p_i$.

\item[Step 2] If $l^{\lfloor d/n \rfloor + 1} \mid f$ or $f = \sum_{i = 0}^{\lceil d/n \rceil - 1} l^i p_i$, then $f$ is unstable. Else, proceed to Step 3.

\item[Step 3] If $n \nmid d$, then $f$ is polystable. Further, in this case, if $\dim(\SL(V)_f) = 0$, then $f$ is stable. If $n \mid d$, proceed to Step 4.

\item[Step 4] Check if $l^{d/n} \mid f$ or $f = \sum_{i = 0}^{d/n} l^i p_i$. If neither holds, then $f$ is polystable. Further, in this case, if $\dim(\SL(V)_f) = 0$, then $f$ is stable. If one of  $l^{d/n} \mid f$ or $f = \sum_{i = 0}^{d/n} l^i p_i$ hold, then go to Step 5.

\item[Step 5] Let $f' = l^{d/n} f_{d/n}$. If $\dim(\SL(V)_{f'}) = \dim(\SL(V)_f)$, then $f$ is polystable and in this case if $\dim(\SL(V)_f) = 0$, then $f$ is stable. If $\dim(\SL(V)_{f'}) \neq \dim(\SL(V)_f)$, then $f$ is semistable, but not polystable.

\end{description}
\end{algorithm}

Most of the steps in the above algorithm are fairly straightforward from an algorithmic perspective, especially since we do not worry about complexity issues. The only non-trivial step is the computation of $\dim(\SL(V)_f)$ and $\dim(\SL(V)_{f'})$. These can be computed by Gr\"obner basis techniques, see \cite[Chapter~9]{Cox-Little-Oshea}. In chacteristic $0$, these can actually be computed by computing the dimensions of their Lie algebras, which is a linear algebraic computation.

\subsection{The case $p \mid n$:}
Recall from Corollary~\ref{cor:sym.flags} that there are essentially four possible flags of $S_n$-stable subspaces. It is easy to observe that $0 \subset L \subset M \subset K^n$ refines all such flags. We can take advantage of this fact to reduce the problem of testing polystability for a symmetric polynomial to a problem on a $2$-dimensional torus (and a computation of the stabilizer).

\begin{lemma} \label{lem:pmidn}
Suppose that $\kar(K) = p \mid n$. Let $0 \neq f \in V = K[x_1,\dots,x_n]_d^{S_n}$. Let $\mathcal{B} = (l,b_1,b_2,\dots,b_{n-2},c)$ be a basis of $K^n$, where $l = x_1 + \dots + x_n$, let $b_i = x_i - x_{i+1}$ for $i = 1,2,\dots,n-2$ and let $c = x_n$. Let $T_2 = (K^*)^2$ denote the two-dimensional torus acting on $K^n$ by $t \cdot l = t_1 l$, $t \cdot b_i = t_2 b_i$ for all $i$ and $t \cdot c = t_1^{-1} t_2^{-(n-2)} c$. Let $w = {\rm ess}(f)$ denote a point in the unique closed orbit of $\overline{O_{T_2,f}}$. Then 
\begin{itemize}
\item $f$ is polystable if and only if $\dim((\SL_n)_f) = \dim((\SL_n)_w)$;
\item $f$ is semistable if and only if $w \neq 0$.
\end{itemize}
\end{lemma}

\begin{proof}
Let $S$ be a non-empty closed $\SL_n$-stable subset of $V$ such that $S \cap O_{\SL_n,f} = \emptyset$ and $S \cap \overline{O_{\SL_n,f}} \neq \emptyset$. Then, let $P_{S,f}$ be the optimal parabolic subgroup. Now $P_{S,f} = P_\mathcal{F}$ for some flag $\mathcal{F}$ of $S_n$-stable subspaces of $K^n$. By Corollary~\ref{cor:sym.flags}, there are four possibilities:\begin{itemize}
\item $0 \subset K^n$;
\item $0 \subset L \subset K^n$;
\item $0 \subset M \subset K^n$;
\item $0 \subset L \subset M \subset K^n$.
\end{itemize} 

The first is ruled out by the same argument as in Lemma~\ref{pnmidn-1-onepsg}. For the second flag, a splitting of $K^n$ is given by $\underline{G} = (L, B \oplus C)$, where $L$ is the span of $l$, $B$ is the span of $b_i$ for $1 \leq i \leq n-2$ and $C$ is the span of $c$. Let $\lambda$ be the optimal $1$-parameter subgroup associated to the splitting $\underline{G}$. Then $\lambda(t) v = t^{c_1}v$ for $v \in L$ and $\lambda(t) v = t^{c_2} v$ for $v \in B \oplus C$ such that $c_1 + (n-1) c_2 = 0$. In particular, $\lambda \in T_2$, so this means that $S \cap \overline{O_{T_2,f}} \neq \emptyset$. A similar argument holds for the other two possibilities of flags. Hence, in any case, we must have $S \cap \overline{O_{T_2,f}} \neq \emptyset$.

To summarize, suppose we have a closed $\SL_n$-stable subset $S$ such that $S \cap O_{\SL_n,f} = \emptyset$ and $S \cap \overline{O_{\SL_n,f}} \neq \emptyset$, then $S \cap \overline{O_{T_2,f}} \neq \emptyset$. Now, since $T_2$ is a torus, we get that $S \cap \overline{O_{T_2,f}} \neq \emptyset$ if and only if $w \in S$.

Now, take $S = \overline{O_{\SL_n, f}} \setminus O_{\SL_n,f}$. Thus, $S \neq \emptyset \Leftrightarrow w \in S$. Clearly, $w \in \overline{O_{\SL_n,f}}$, so $\dim((\SL_n)_f) \geq \dim((\SL_n)_w)$. Thus $w \in S \Leftrightarrow \dim((\SL_n)_f) > \dim((\SL_n)_w) \Leftrightarrow \dim((\SL_n)_f) \neq \dim((\SL_n)_w)$.  Thus $f$ is polystable $\Leftrightarrow S = \emptyset \Leftrightarrow \dim((\SL_n)_f) = \dim((\SL_n)_w)$.

The argument for semistability is analogous where you take $S = \{0\}$ instead of $\overline{O_{\SL_n, f}} \setminus O_{\SL_n,f}$.

\end{proof}

\begin{algorithm} \label{algo:pmidn}
 Now we give an algorithm that decides whether a symmetric polynomial is unstable/semistable/polystable/stable in the case $p \mid n$.\\
\noindent{\bf Input:}  $f \in K[x_1,\dots,x_n]_d$
\begin{description}
\setlength\itemsep{.5em}
\item [Step 1] Compute $w = {\rm ess}(f)$ as in Lemma~\ref{lem:pmidn}. If $w = 0$, then $f$ is unstable. Else, proceed to Step 2

\item[Step 2] If $\dim((\SL_n)_f) \neq \dim((\SL_n)_w)$, then $f$ is semistable, not polystable. Else $f$ is polystable. Moreover, in the case that $f$ is polystable, $\dim(\SL_n)_f = 0$ if and only if $f$ is stable. 

\end{description}
\end{algorithm}

\begin{proof} [Proof of Theorem~\ref{thm:symm-algo}]
Thus follows from Algorithms~\ref{algo:pnmidn} and ~\ref{algo:pmidn}.
\end{proof}

\section{Polystability for interesting classes of symmetric polynomials} \label{sec:symm-poly-class} 
We first briefly recall important results on symmetric polynomials, using the opportunity to introduce notation. While symmetric polynomials in characteristic zero is widely studied, the case of positive characteristic receives far less attention, so we will be particularly careful about characteristic assumptions.

First, we define elementary symmetric functions. For each $1 \leq k \leq n$, we define the $k^{th}$ elementary symmetric polynomial
$$
e_k(x_1,\dots,x_n) = \sum_{1 \leq i_1 < i_2 < \dots < i_k \leq n} x_{i_1} x_{i_2} \cdots x_{i_k}
$$

We also define the $k^{th}$ homogeneous symmetric polynomial
$$
h_k(x_1,\dots,x_n) = \sum_{1 \leq i_1 \leq i_2 \leq \dots \leq i_k \leq n} x_{i_1} x_{i_2} \cdots x_{i_k}
$$

Let $\Lambda(n) = K[x_1,\dots,x_n]^{S_n}$ denote the ring of symmetric polynomials. The collection $\{e_k(x_1,\dots,x_n)\ |\ 1 \leq k \leq n\}$ forms an algebraically independent set of generators for $\Lambda(n)$ as does the collection $\{h_k(x_1,\dots,x_n)\ | \ 1 \leq k \leq n\}$. In characteristic zero $\{p_k(x_1,\dots,x_n)\ | \ 1 \leq k \leq n\}$ forms an algebraically independent set of generators as well where $p_k$ denotes the power sum symmetric polynomial
$$
p_k(x_1,\dots,x_n) = x_1^k + x_2^k + \dots + x_n^k.
$$

However, power sum symmetric polynomials do not form a generating set if $\kar(K) < n$.

For each partition $\lambda = (\lambda_1,\dots,\lambda_l) \vdash d$, we define
\begin{align*}
e_\lambda &= e_{\lambda_1} e_{\lambda_2} \cdots e_{\lambda_l} \\
h_\lambda & = h_{\lambda_1} h_{\lambda_2} \cdots h_{\lambda_l} \\
p_\lambda & = p_{\lambda_1} p_{\lambda_2} \cdots p_{\lambda_l}
\end{align*}

The collection $\{e_\lambda(x_1,\dots,x_n)\ | \lambda \vdash d\}$ forms a linear basis for $\Lambda(n)_d$, the space of degree $d$ symmetric polynomials as does $\{h_\lambda(x_1,\dots,x_n)\ | \lambda \vdash d\}$ and in characteristic $0$ $\{p_\lambda(x_1,\dots,x_n)\ | \lambda \vdash d\}$ forms a basis as well. In particular, $\dim_K(\Lambda(n)_d)$ is equal to the number of partitions of $d$. A very straightforward way to see this is to define monomial symmetric functions. We say an exponent vector $e = (e_1,\dots,e_n)$ is of type $\lambda$ if it is a permutation of $(\lambda_1,\dots,\lambda_n)$ where we add trailing zeros to $\lambda$ if it does not have sufficient parts. 
$$
m_\lambda(x_1,\dots,x_n) = \sum_{e = (e_1,\dots,e_n) \text{ of type } \lambda} x^e.
$$

It is entirely obvious that $\{m_\lambda\ |\ \lambda \vdash d\}$ is a linear basis of $\Lambda(n)_d$.

Another interesting collection of symmetric polynomials are the Schur polynomials whose importance comes from the representation theory of the symmetric group (or equivalently the general linear group). For $\lambda = (\lambda_1,\dots,\lambda_l) \vdash d$, we define the Schur polynomial

$$
s_\lambda(x_1,\dots,x_n) = \det \begin{pmatrix} h_{\lambda_1} & h_{\lambda_1 + 1} & \dots  & \dots  & h_{\lambda_1 + l-1} \\ h_{\lambda_2 - 1} & h_{\lambda_2} & h_{\lambda_2 + 1} &\dots  & h_{\lambda_2 + l - 2} \\    && \ddots &&\vdots  \\ &&& \ddots &\vdots   \\  && \dots & h_{\lambda_l - 1} &  h_{\lambda_l} \end{pmatrix},
$$
where $h_d = 0$ for $d < 0$ and $h_0 = 1$.
In particular, $s_{(1^d)}(x_1,\dots,x_n) = e_d (x_1,\dots,x_n)$ and $s_{d}(x_1,\dots,x_n) = h_d(x_1,\dots,x_n)$. There are other equivalent definitions of Schur functions and we will recall them as and when we need them.

Recall that $f \in K[x_1,\dots,x_n]_d^{S_n}$, we write $f$ as a polynomial in $l = \sum_i x_i$ and $b_1,\dots,b_{n-1}$ where $b_i = x_i - x_{i+1}$. In other words, we have $f = \sum_{i=0}^d l^i p_i$, where $p_i$ is a polynomial in $b_1,b_2,\dots,b_{n-1}$. Let $D = \sum_i \frac{\partial}{\partial_{x_i}}$ for this section.

\begin{lemma} \label{lem:diff-operator-p}
Assume $p \nmid n$. Let $f \in K[x_1,\dots,x_n]_d$. Then $f \in K[b_1,\dots,b_{n-1}]_d$ if and only if $\frac{D^k}{k!} f = 0$ for all $k \in \Z_{>0}$.
\end{lemma} 
 
\begin{proof}
Suppose $f \in K[b_1,\dots,b_{n-1}]_d$, then $Db_i = 0$, so $\frac{D^k}{k!} f = 0$ for all $k \in \Z_{>0}$. Conversely, suppose $f \notin K[b_1,\dots,b_{n-1}]_d$. Then, write $f = \sum_i l^i p_i$, and there exists $j > 0$ is such that $p_j \neq 0$. Then, clearly $\frac{D^j}{j!} f \neq 0$.
\end{proof}

\begin{remark}
In the above lemma, dividing by $k!$ may not make sense in characteristic $p$ if $k$ is large enough. Yet, the differential operator $\frac{D^k}{k!}$ is well-defined. This is standard and we leave the details to the reader.
\end{remark}

In characteristic zero, we have a stronger statement.

\begin{lemma} \label{lem:diff-operator-0}
Let $\kar(K) = 0$. Let $f \in K[x_1,\dots,x_n]_d$. Then $f \in K[b_1,\dots,b_{n-1}]_d$ if and only if $D f = 0$.
\end{lemma}
 
\begin{proof}
This is similar to Lemma~\ref{lem:diff-operator-p} and we leave it to the reader.
\end{proof}

\begin{lemma} \label{lem:compute-D}
We have 
\begin{itemize}
\item $D e_k(x_1,\dots,x_n) = (n+1-k) e_{k-1} (x_1,\dots,x_n)$;
\item $D h_k(x_1,\dots,x_n) = (n+k-1) h_{k-1}(x_1,\dots,x_n)$;
\item $D p_k(x_1,\dots,x_n) =  k p_{k-1} (x_1,\dots,x_n)$.
\end{itemize}
\end{lemma}

\begin{proof}
This is a straightforward computation and is left to the reader.
\end{proof}

\subsection{Elementary, homogeneous and power sum symmetric polynomials}
We first state a lemma
\begin{lemma} \label{lem:interesting-0}
Suppose $\kar(K) = 0$, $\lambda \vdash d$ and $d < n$. Let $f \in K[x_1,\dots,x_n]_d^{S_n}$. Then $f$ is either unstable or polystable. Further, $f$ is polystable if and only if $l \nmid f$ and $Df \neq 0$.
\end{lemma}

\begin{proof}
Follows from Theorem~\ref{theo:pnmidn} and Lemma~\ref{lem:diff-operator-0}.
\end{proof}

\begin{proposition} \label{prop:interesting-0}
Let $\lambda \vdash d$ be a partition and let $d < n$. Assume $\kar(K) = 0$. Then $e_\lambda(x_1,\dots,x_n)$, $h_\lambda(x_1,\dots,x_n)$ and $p_\lambda(x_1,\dots,x_n)$ are either polystable or unstable. Further, they are polystable if and only if all non-zero parts of $\lambda$ are $\geq 2$.  
\end{proposition}

\begin{proof}
First, let us consider $e_\lambda$'s. We see that $e_\lambda = e_{\lambda_1} e_{\lambda_2} \cdots e_{\lambda_l}$. Thus $l$ divides $e_\lambda$ if and only if $l$ divides $e_{\lambda_i}$ for some $i$. But now, we see that $l = e_1, e_2,\dots,e_n$ are algebraically independent, so $l = e_1$ divides $e_{\lambda_i}$ if and only if $\lambda_i = 1$. A similar argument holds for $h_\lambda$ and $p_\lambda$. Thus, to summarize, we conclude that $l$ does not divide $e_\lambda/h_\lambda/p_\lambda$ if and only if every non-zero part of $\lambda$ is at least $2$.

Now, consider the action of $D$ on $e_\lambda = e_{\lambda_1} e_{\lambda_2} \cdots e_{\lambda_l}$. We see that $De_\lambda = \sum_{i=1}^l (n + 1- \lambda_i) e_{\lambda_1} \cdots e_{\lambda_i - 1} \cdots e_{\lambda_l} \neq 0$. since $n+1 - \lambda_i > 0$ for all $\lambda_i$ since $\lambda_i \leq d < n$. Similarly, $D h_\lambda \neq 0$ and $D p_\lambda \neq 0$. 

Now, the proposition follows by applying Lemma~\ref{lem:interesting-0}.

\end{proof}

\begin{lemma} \label{lem:interesting-p}
Suppose $\kar(K) = p > 0$, $p \nmid n$, $\lambda \vdash d$ and $d < n$. Let $f \in K[x_1,\dots,x_n]_d^{S_n}$. Then $f$ is polystable if and only if $l \nmid f$ and $\frac{D^k}{k!} f \neq 0$ for some $k \in \Z_{>0}$.
\end{lemma}

\begin{proof}
Follows from Theorem~\ref{theo:pnmidn} and Lemma~\ref{lem:diff-operator-p}.
\end{proof}

\begin{proposition} \label{prop:elementary-p}
Assume $\kar(K) = p \nmid n$. Let $\lambda = k_1^{a_1} k_2^{a_2} \dots k_l^{a_l} \vdash d$ be a partition and let $d < n$. Then $e_\lambda(x_1,\dots,x_n)$ is polystable if the following conditions hold:
\begin{itemize}
\item Every non-zero part of $\lambda$ is $\geq 2$;
\item $p \nmid (n+1-k_i) a_i$ for some $i$. 
\end{itemize}
\end{proposition}

\begin{proof}
As in the proof of Proposition~\ref{prop:interesting-0}, we can show that $l \nmid e_\lambda(x_1,\dots,x_n)$ if and only if every non-zero part of $\lambda$ is $\geq 2$ (since $e_1,\dots,e_n$ are algebraically independent even in positive characteristic). The condition $p \nmid (n+1-k_i) a_i$ for some $i$ ensures that $D e_\lambda \neq 0$ by the same computation as in the proof of Proposition~\ref{prop:interesting-0}. The proposition then follows from  Lemma~\ref{lem:interesting-p}.
\end{proof}

\begin{proposition}
Assume $\kar(K) = p \nmid n$. Let $\lambda = k_1^{a_1} k_2^{a_2} \dots k_l^{a_l} \vdash d$ be a partition and let $d < n$. Then $h_\lambda(x_1,\dots,x_n)$ is polystable if the following conditions hold:
\begin{itemize}
\item Every non-zero part of $\lambda$ is $\geq 2$;
\item $p \nmid (n + k_i - 1) a_i$ for some $i$. 
\end{itemize}
\end{proposition}

\begin{proof}
Similar to Proposition~\ref{prop:elementary-p} and left to the reader.
\end{proof}

\begin{proposition}
Assume $\kar(K) = p \nmid n$. Let $\lambda = k_1^{a_1} k_2^{a_2} \dots k_l^{a_l} \vdash d$ be a partition and let $d < n$. Then $p_\lambda(x_1,\dots,x_n)$ is polystable if the following conditions hold:
\begin{itemize}
\item No part of $\lambda$ is of equal to $p^c$ for some $c \in \Z_{\geq 0}$;
\item $p \nmid a_i k_i$ for some $i$. 
\end{itemize}
\end{proposition}


\begin{proof}
This is also similar to Proposition~\ref{prop:elementary-p}. The only difference is that for $k \in \Z_{>0}$ such that $k \leq d < n$, we have $l \mid p_k(x_1,\dots,x_n)$ if and only if $k = p^c$ for some $c$, which one sees by the following brief argument.

First, if $n = 2$, then $d = k = 1$ is the only case to check. In this case $k = p^0 = 1$ and $l = p_1(x_1,\dots,x_n)$, so clearly $l \mid p_1(x_1,\dots,x_n)$. Now, we assume $n \geq 3$. Clearly if $k = p^c$, then $l \mid p_k(x_1,\dots,x_n)$. On the other hand, suppose $l = x_1 + \dots + x_n \mid p_k(x_1,\dots,x_n)$. Then, we have $x_1 + x_2 + x_3 \mid p_k(x_1,x_2,x_3)$ by setting $x_4 = x_5 = \dots = x_n = 0$ (this step requires $n \geq 3$). Since setting $x_3 = - (x_1 + x_2)$ kills a divisor of $p_k(x_1,x_2,x_3)$, namely $x_1 + x_2 + x_3$, we conclude that $p_k(x_1,x_2, - (x_1 + x_2)) = 0$. This means that $x_1^k + x_2^k + (-1)^k(x_1 + x_2)^k = 0$. Now, suppose $k$ is not a power of $p$, then there exists $1 \leq i \leq k-1$ such that ${k \choose i} \neq 0$,\footnote{Indeed, if we write $k = dp^e$ where $d \geq 2$ is coprime to $p$, then ${k \choose p^e} \neq 0$ in characteristic $p$.} so this means that when you expand out $x_1^k + x_2^k + (-1)^k(x_1 + x_2)^k$ as a sum of monomials, we have the non-zero term $(-1)^k {k \choose i} x_1^k x_2^{k-i}$, which contradicts $x_1^k + x_2^k + (-1)^k(x_1 + x_2)^k = 0$. Hence $k$ is a power of $p$.
\end{proof}

\begin{remark}
Since we know how to compute $D f$ when $f = e_\lambda, h_\lambda$ and $p_\lambda$, we can always compute $\frac{D^r}{r!} f$ and check if it is non-zero for some $r$. 
\end{remark}

\subsection{Schur polynomials}
The case of Schur polynomials is a little more tricky. We need a few preparatory lemmas.

\begin{lemma} \label{lem:l-schur}
Let $\kar(K) = p \nmid n$, let $\lambda \vdash d$, and suppose $1 < d < n$. Then, we have $l \nmid s_\lambda$.
\end{lemma}

\begin{proof}
Recall that $l = s_\mu$ where $\mu = (1)$. For $t$, let us denote by $P_t$ the collection of all partitions of size $t$. Then, for $t < n$, one checks that $\{s_\lambda\ | \ \lambda \vdash t\}$ is a basis for $K[x_1,\dots,x_n]_t^{S_n}$ as follows. First, it is clear that $\{h_\lambda\ | \ \lambda \vdash t\}$ is a basis. Now, by the definition of Schur polynomials, one sees that $s_\lambda = h_\lambda + \sum_{\mu \succ \lambda } c_{\mu,\lambda} h_\mu$ for some constants $c_{\mu,\lambda}$. Here $\succ$ denotes the lexicographic order. Thus, the linear transformation that sends $h_\lambda \mapsto s_\lambda$ is unipotent and hence invertible. Thus, we conclude $\{s_\lambda\ | \ \lambda \vdash t\}$ is also a basis. 

Now, if $l \mid s_\lambda$, then $s_\lambda = l \cdot f$ where $f \in K[x_1,\dots,x_n]_{d-1}^{S_n}$. Thus, we can write $f = \sum_{\nu \vdash d-1} a_\nu s_\nu$. But then we can compute $l \cdot f$ by the Pieri rule. We know that $l \cdot s_\nu = \sum_{\mu = \nu \cup \text{ one box}}  s_\mu$. Now, let $S = \{\nu\ |\ a_\nu \neq 0\}$. Then, under the dominance order, let $\widetilde{\nu}$ be a maximal element and $\overline{\nu} $ be a minimal element in $S$. Then, when we write $l \cdot f$ as a linear combination of Schur polynomials, we see by the Pieri rule that that the coefficient of $s_{\widetilde{\nu} + (1,0,\dots,0)}$ is $a_{\widetilde{\nu}} \neq 0$ and that the coefficent of $s_{(\overline{\nu}_1, \overline{\nu}_2, \dots, \overline{\nu}_r, 1)}$ is $a_{\overline{\nu}} \neq 0$ (where $r$ is the number of non-zero parts of $\overline{\nu}$). Thus $l \cdot f$ when written as a linear combination of Schur polynomials contains at least two terms, so we cannot have $l \cdot f = s_\lambda$. Thus $l \nmid s_\lambda$.

\end{proof}

We point out that in the above argument, it is crucial that $d-1 > 0$, since otherwise, we would have $\overline{\nu} = \widetilde{\nu} = \emptyset$ and ${\widetilde{\nu} + (1,0,\dots,0)} = (\overline{\nu}_1, \overline{\nu}_2, \dots, \overline{\nu}_r, 1) = (1)$, so we would not be able to get a contradiction. This is perfectly reasonable since if $d = 1$, we have $s_\lambda = s_{(1)} = l$, so of course $l \mid s_\lambda$. We also point out that if $d > n$, then some of the $s_\lambda$'s will be zero, so $\{s_\lambda\ | \ \lambda \vdash t\}$ will not be a linearly independent set anymore.

The next computation we need is to understand the action of the differential operator $D$ on $s_\lambda$. For a partition $\lambda$, we identify it with its Young diagram, where the boxes are indexed with matrix coordinates. Thus, we have $(i,j) \in \lambda$ if the $i^{th}$ row of $\lambda$ is at least $j$, i.e., $\lambda_i \geq j$. For $(i,j) \in \lambda$, we write $d_{(i,j)} = j-i$. When $\lambda, \mu$ are two partitions such that $\lambda$ is obtained from $\mu$ by adding a box in position $(i,j)$, then we write $d_{\lambda \setminus \mu} = d_{(i,j)} = j- i$.

\begin{proposition} \label{prop:D-schur}
Let $\lambda \vdash d$ be a partition, and let $d < n$. Then 
$$
D(s_\lambda(x_1,\dots,x_n)) = \sum_{\lambda = \mu \cup \text{ one box}} (n+ d_{\lambda \setminus \mu}) s_{\mu}(x_1,\dots,x_n),
$$
\end{proposition}

\begin{proof}
For $\alpha = (\alpha_1,\dots,\alpha_n)$, we define $a_\alpha = \det(x_i^{\alpha_j + n-j})_{1 \leq i,j \leq n}$. Let $\delta = (n-1,n-2,\dots,1,0)$. Then, it is well known that 
$$
s_\lambda = \frac{a_{\lambda + \delta}}{a_\delta}.
$$

One easily checks that $D a_\alpha = \sum_i \alpha_i a_{\alpha - \mathbb{1}_i}$ where $\mathbb{1}_i$ is a vector with a $1$ in it's $i^{th}$ spot and $0$'s everywhere else. Moreover, one observes that $a_\beta  =0$ if and only if $\beta_i = \beta_j$ for some $i \neq j$. With these two observations, we compute
$$
D a_{\lambda + \delta} = \sum_{\lambda = \mu \cup \text{ one box}} (n + d_{\lambda \setminus \mu}) a_{\mu + \delta}.
$$

We also observe that $D a_\delta = 0$. These two computations, along with the formula for $s_\lambda$ yield
$$
D s_\lambda = \sum_{\lambda = \mu \cup \text{ one box}} (n + d_{\lambda \setminus \mu}) s_\mu.
$$

\end{proof}

If $\mu \subseteq \lambda$, we define $M(\lambda \setminus \mu, n) = \prod_{(i,j) \in \lambda \setminus \mu} (n + j -i)$. Further, let $f^{\lambda \setminus \mu}$ denote the number of standard Young tableau of skew shape $\lambda \setminus \mu$. Then, from the above proposition, one deduces:

\begin{corollary}
Let $\lambda \vdash d$ be a partition and $d < n$. Then 
$$
\frac{D^k}{k!} (s_\lambda(x_1,\dots,x_n)) = \sum_{\mu \subset \lambda, |\lambda \setminus \mu| = k} \frac{M(\lambda \setminus \mu)}{k!} \cdot f^{\lambda \setminus \mu} \cdot s_\mu(x_1,\dots,x_n).
$$
\end{corollary}

We can now prove Theorem~\ref{thm:Schur}.

\begin{proof} [Proof of Theorem~\ref{thm:Schur}]
We have $l \nmid s_\lambda$ by Lemma~\ref{lem:l-schur} and $Ds_\lambda \neq 0$ by Proposition~\ref{prop:D-schur}. Hence, the corollary follows from Lemma~\ref{lem:interesting-0}. 
\end{proof}

Finally, we note that in positive characteristic, we need to be able to check when $\frac{D^k}{k!} s_\lambda \neq 0$. This is equivalent to checking if $\frac{M(\lambda \setminus \mu,n)}{k!} \cdot f^{\lambda \setminus \mu} \neq 0$ for some $\mu \subseteq \lambda$. We know how to compute $M(\lambda \setminus \mu,n)$, so it suffices to know how to compute $f^{\lambda \setminus \mu}$. A formula for that was given by Aitken \cite{Aitken} (rediscovered by Feit \cite{Feit}]), see also \cite[Corollary~7.16.3]{Stanley2}.

\begin{theorem} [Aitken, Feit]
Let $\mu \subseteq \lambda$ be partitions and suppose $l(\lambda) \leq N$. Then, 
$$
f^{\lambda \setminus \mu} = N! \cdot \det \left[ \frac{1}{(\lambda_i - \mu_j - i + j)!}\right]_{i,j=1}^N
$$
\end{theorem}

Thus, even in positive characteristic, for any specific $\lambda$, using these techniques one should be able to determine whether $s_\lambda$ is polystable or not in the case $p \nmid n$ and $d < n$.

\begin{remark}
In this section, we presented a series of results on polystability of various interesting families of symmetric polynomials, in particular demonstrating the effectiveness of our approach for proving polystability. Our approach provides a systematic approach to proving many more such results, some of which might require interesting combinatorial results to establish.  
\end{remark}

\subsection{Acknowledgements}
The first author was partially supported by NSF grants IIS-1837985 and DMS-2001460. The second author was partially supported by the
University of Melbourne and by NSF grants DMS-1638352 and CCF-1900460. 

We would like to thank Michel Brion, Gurbir Dhillon, Pierre Deligne, Nate Harman, Victor Ginzburg, Robert Guralnick, Victor Kac, Allen Knutson, Daniel Litt, Ivan Losev, Rohit Nagpal, Abhishek Oswal, Alexander Premet, Peter Sarnak, and Akshay Venkatesh for interesting discussions.

\bibliographystyle{alpha}
\bibliography{references}

\appendix
\section{Proof of Theorem~\ref{thm:complete-homog-3vars}}
We now prove Theorem~\ref{thm:complete-homog-3vars}. One can simply implement the algorithms outlined in this paper on a computer to verify this result (although it needs a little bit more effort than naively implementing the algorithm because we want to determine polystability for all primes, which is apriori an infinite set of computations). However, we will not directly appeal to the algorithms and instead give an explicit argument. This has a few advantages. First, it demonstrates the flexibility we actually have in using the ideas developed in this paper. Second, we want to make the computations as manageable as possible, i.e., even though we omit many of the computational details, we intend for it to be hand checkable by the reader with sufficient (but not unearthly) patience. Indeed, we did these computations by hand. Finally, we want to illustrate the flavor of combinatorial computations one encounters, and we hope that a deeper analysis of the combinatorics involved can lead to a better understanding of polystability for interesting classes of symmetric polynomials, beyond what we discussed in Section~\ref{sec:symm-poly-class}.

\begin{proof} [Proof of Theorem~\ref{thm:complete-homog-3vars}]
First, recall that $h_3(x,y,z)$ is the sum of all degree $3$ monomials in $x,y$ and $z$, so
$$
h_3(x,y,z) = x^3 + y^3 + z^3 + x^2 y + x^2 z+ xy^2 +xz^2 + yz^2 + y^2z + xyz.
$$

\begin{itemize}

\item {\bf Case 1: $p \nmid (3 = n)$, i.e., $p \neq 3$:}

Suppose $h_3(x,y,z)$ is not polystable. Then, the optimal parabolic subgroup must either be $\mathcal{F} = 0 \subseteq L \subseteq K^3$ or $\mathcal{G} = 0 \subseteq M \subseteq K^3$, where $L = \spa (x+y+z)$ and $M = \spa (x-y, y-z)$ by Corollary~\ref{cor:sym.flags}.

Suppose $\mathcal{F}$ is the optimal parabolic. Then, a compatible basis is $\mathcal{B} = (t,q,r)$, where $t = x + y + z, q = y, r = z$. Let $T_\mathcal{B}$ be the corresponding torus. We compute the change of basis:
$$
h_3(x,y,z) = h_3(t-q-r, q,r) = t^3 - 2t^2q - 2t^2 r + 2tq^2 + 3tqr + 2tr^2 - q^2 r - qr^2.
$$
Recall that we are in the case $p \neq 3$. When $p \neq 2$, the Newton polytope for $h_3(x,y,z)$ with respect to the torus $T_{t,q,r}$ is:

$$
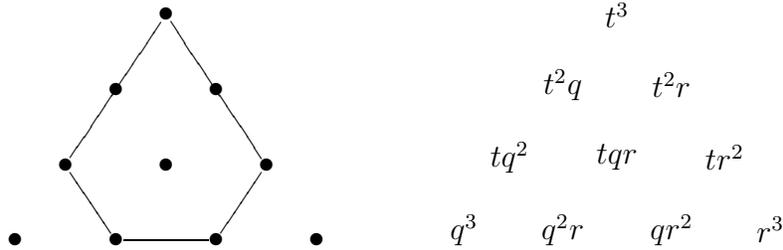
\xymatrix @C=1pc @R=1.732pc{
&&& *{\bullet}\ar@{-}[lldd]\ar@{-}[rrdd] &&&\\
&&*{\bullet} && *{\bullet} &&\\
&*{\bullet}\ar@{-}[rd] && *{\bullet} && *{\bullet}\ar@{-}[ld] &\\
*{\bullet} && *{\bullet}\ar@{-}[rr] && *{\bullet} && *{\bullet}}
\qquad\qquad
\xymatrix @C=0.4pc @R=1.3pc{
&&& *{t^3} &&&\\
&&*{t^2q} && *{t^2r} &&\\
&*{tq^2} && *{tqr} && *{tr^2} &\\
*{q^3} && *{q^2r} && *{qr^2} && *{r^3}}
$$

%
%
%
%

The picture on the right gives the dictionary between the monomials and their weights. Note that the weight of a monomial is just its exponent vector, so weight of $t^3$ is $(3,0,0)$, weight of $tr^2$ is $(1,0,2)$, etc. Now, by Corollary~\ref{cor:symm-torus}, we conclude that $h_3(x,y,z)$ is $T_\mathcal{B}$-polystable since $(1,1,1)$ is in the relative interior of the Newton polytope.

When $p = 2$, the above simplifies to:
$$
h_3(x,y,z) = h_3(t-q-r, q,r) = t^3 + tqr + q^2 r + qr^2.
$$

Now, the Newton polytope is a convex hull of 4 points, and it is easily seen that $(1,1,1)$ is again in the interior of the Newton polytope, picture below:
$$
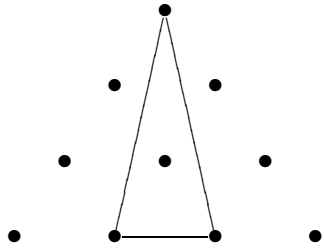
\xymatrix @C=1pc @R=1.732pc{
&&& *{\bullet}\ar@{-}[lddd]\ar@{-}[rddd] &&&\\
&&*{\bullet} && *{\bullet} &&\\
&*{\bullet}&& *{\bullet} && *{\bullet} &\\
*{\bullet} && *{\bullet}\ar@{-}[rr] && *{\bullet} && *{\bullet}}
$$





Hence, we get that the $T_\mathcal{B}$ orbit of $h_3(x,y,z)$ is closed. Thus for all $p$ such that $p \nmid n$, $h_3(x,y,z)$ is $T_\mathcal{B}$-polystable, so $\mathcal{F}$ cannot be an optimal parabolic subgroup by Theorem~\ref{thm:Kempf}.

Now, suppose $\mathcal{G}$ is the optimal parabolic subgroup. Then, a compatible basis is $t = x - y$, $q = y-z$ and $r = z$. Write $\mathcal{B} = (t,q,r)$ and let $T_\mathcal{B}$ be the corresponding torus. We compute the change of basis:
\begin{align*}
h_3(x,y,z) & = h_3(t+q+r, q+r, r) \\
& = t^3 + 4t^2 q + 5t^2r + 6tq^2 + 15tqr + 10tr^2 + 4 q^3 + 15 q^2r + 20 qr^2 + 10 r^3.
\end{align*}

Unless $p = 2$ or $p = 5$, it is easy to conclude that $h_3(x,y,z)$ is polystable with respect to $T_\mathcal{B}$ by computing its Newton polytope and we leave the details to the reader. On the other hand, when $p = 2$, we have:
$$
h_3(x,y,z) = t^3 +  t^2r + tqr + q^2r.
$$

The Newton polyope is:

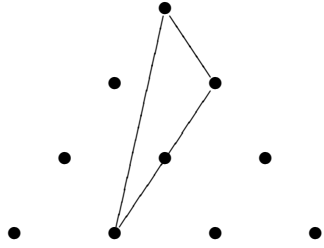
$$
\xymatrix @C=1pc @R=1.732pc{
&&& *{\bullet}\ar@{-}[lddd]\ar@{-}[rd] &&&\\
&&*{\bullet} && *{\bullet} &&\\
&*{\bullet}&& *{\bullet} && *{\bullet} &\\
*{\bullet} && *{\bullet}\ar@{-}[rruu] && *{\bullet} && *{\bullet}}
$$

%

As is evident, the point $(1,1,1)$ is not in the relative interior, so $h_3(x,y,z)$ is not $T_\mathcal{B}$-polystable. Thus, it suffices to check if $w = {\rm ess}(h_3(x,y,z))$ (with respect to $T_\mathcal{B}$) is in the $\SL_3$-orbit of $h_3(x,y,z)$. One easily computes $w = t^2r + tqr + q^2r = r(t^2 + tq + q^2)$, which is reducible. But $h_3(x,y,z) = t^3 +  t^2r + tqr + q^2r$ is irreducible -- think of it as a polynomial in the variable $t$ with coefficients in the PID $K(q)[r]$ and apply Eisenstein's criterion with the prime $r$. Thus $h_3(x,y,z)$ and $w$ are not in the same orbit.



Thus, to summarize, for $p = 2$, $h_3(x,y,z)$ is not $\SL_3$ polystable, $\mathcal{G}$ is an optimal parabolic subgroup and $w = t^2r + tqr + q^2r = r(t^2 + tq + q^2)$ is a point in the boundary of the $\SL_3$ orbit of $h_3(x,y,z)$.

Now, the case of $p = 5$. In this case, we have
$$
h_3(x,y,z) = t^3 + 4t^2 q + 6tq^2 +  4 q^3.
$$

We omit the details, but one can check by a similar analysis as above that $t^3 + 4t^2 q + 6tq^2 +  4 q^3$ is actually unstable with respect to $T_\mathcal{B}$. So, $h_3(x,y,z)$ is $\SL_3$ unstable (and in particular not polystable) when $p = 5$! \\
 
\item {\bf Case 2: The case $p = 3$:}
We will be brief with this case. Suppose $h_3(x,y,z)$ is not polystable, then there are three possible choices for optimal parabolic. However, the basis $t = x+y+z, q = y-z, r = z$ is compatible with all possible optimal parabolics. Thus, if we check that $h_3(x,y,z)$ is $T_\mathcal{B}$ polystable, where $\mathcal{B} = (t,q,r)$, then we get a contradiction, so $h_3(x,y,z)$ must be polystable. 

We compute the change of basis
\begin{align*}
h_3(x,y,z) &= h_3(t+2q+r, q+r, r) \\
& = t^3 + 7t^2 q + 5t^2r + 17tq^2 + 25tqr + 10tr^2 + 15q^3 + 35qr^2 + 30 qr^2 + 10r^3 \\
& = t^3 + t^2 q + 2t^2r + 2tq^2 + tqr + tr^2 + 2qr^2 +  r^3.
\end{align*}

We leave it to the reader to check that $h_3(x,y,z)$ is $T_\mathcal{B}$-polystable by drawing the Newton polytope. \\
\end{itemize}

Thus, we conclude that $h_3(x,y,z)$ is $\SL_3$ polystable unless $p = 2$ or $p = 5$. When $p = 2$, it is 
$\SL_3$ semistable, not $\SL_3$ polystable and perhaps most surpisingly, when $p = 5$, it is $\SL_3$ unstable!
\end{proof}

\end{document}